\documentclass[a4paper]{article}

\usepackage{amsmath,amsthm,amsfonts,latexsym,amscd,amssymb,enumerate,mathtools}
\usepackage[backref]{hyperref}
\usepackage{amsrefs}
\usepackage[all]{xy}
\usepackage{url}
\usepackage{comment}
\newcommand\R{\mathbb{R}}

\newcommand\Q{\mathbb{Q}}
\newcommand\Z{\mathbb{Z}}
\newcommand\integers{\mathbb{Z}}

\newcommand\Jhom{J}
\newcommand\JS{\mathcal{J}}

\newcommand\pr{\textup{pr}}

\newcommand{\iso}{\cong}

\newcommand{\boundary}{\partial}
\newcommand{\id}{\textup{id}}
\newcommand{\reals}{\mathbb{R}}

\newcommand{\Spin}{\textup{Spin}}

\DeclareMathOperator{\Pos}{\mathcal{R}^+_c}
\DeclareMathOperator{\Riem}{\mathcal{R}}
\DeclareMathOperator{\Inv}{\mathcal{R}^{inv}}

\DeclareMathOperator{\relind}{ind-diff}

\newcommand{\into}{\hookrightarrow}

\newcommand\aelt{a}
\newcommand\eot[1]{\prescript{}{2}{#1}}
\newcommand\im{\textup{Im}}

\newcommand\del{\partial}

\newcommand\wt[1]{\widetilde{#1}}

\newcommand\Diff{\textup{Diff}}

\newcommand{\coker}{\textup{coker}}

\newcommand{\xra}{\xrightarrow}

\newcommand{\an}[1]{\langle{#1}{\rangle}}
\newcommand{\wh}{\widehat}
\newcommand{\ol}{\overline}
\newcommand{\colim}{\mathrm{colim}}

\newcommand{\MSpin}{\mathbf{MSpin}}

\newcommand{\KO}{\mathbf{KO}}
\newcommand{\bfS}{\mathbf{S}}
\newcommand{\bfR}{\mathbf{R}}
\swapnumbers
\theoremstyle{plain}
\newtheorem{theorem}{Theorem}[section]
\newtheorem{lemma}[theorem]{Lemma}
\newtheorem{corollary}[theorem]{Corollary}

\theoremstyle{definition}
\newtheorem{definition}[theorem]{Definition}

\theoremstyle{remark}
\newtheorem{remark}[theorem]{Remark}

\usepackage{color}

\newcommand{\remdc}[1]{#1}%{\begingroup\color{blue}#1\endgroup}%{#1}%
\newcommand{\remts}[1]{#1}%{\begingroup\color{red}#1\endgroup}%{#1}%
\newcommand{\remws}[1]{#1}%{\begingroup\color{magenta}#1\endgroup}%{#1}%

\date{\today}

\begin{document}

\title{Harmonic spinors and\\ metrics of positive 
%scalar 
curvature \\via the Gromoll filtration and Toda brackets}
\author{Diarmuid Crowley%
\thanks{
\protect\href{mailto:dcrowley@unimelb.edu.au}{e-mail:
  dcrowley@unimelb.edu.au}
\protect\\
\protect\href{http://www.dcrowley.net/}{www:~http://www.dcrowley.net/}
}\\
School of Mathematics and Statistics, \\University of Melbourne, Australia
\and Thomas Schick%
\thanks{
\protect\href{mailto:thomas.schick@math.uni-goettingen.de}{e-mail:
  thomas.schick@math.uni-goettingen.de}
\protect\\
\protect\href{http://www.uni-math.gwdg.de/schick}{www:~http://www.uni-math.gwdg.de/schick}
}\\
Mathematisches Institut, \\Universit\"at G{\"o}ttingen, Germany
\and Wolfgang Steimle%
\thanks{
\protect\href{mailto:wolfgang.steimle@math.uni-augsburg.de}{e-mail: wolfgang.steimle@math.uni-augsburg.de}
\protect\\
\protect\href{http://www.math.uni-augsburg.de/prof/diff/arbeitsgruppe/steimle}{www:~http://www.math.uni-augsburg.de/prof/diff/arbeitsgruppe/steimle}
}\\
Institut f\"ur Mathematik, Universit\"at Augsburg, Germany
}

\maketitle

%%%%%%%%%%%%%%%%%%%%%%%%%%%%%%%%%%%%%%%%%%%%%%%%
\begin{abstract}
  We construct non-trivial elements of order $2$ in the homotopy groups $\pi_{8j+1+*}
\Diff(D^6,\del)$ {for $*\equiv 1,2 \pmod 8$}, which are detected
through the chain  
  \[\pi_{{8j+1+*}} \Diff(D^6,\del)\to \pi_0 \Diff(D^{{8j+*+7}},\del)\to {KO_{*}}=\integers/2\]
  of the ``assembling homomorphism'' 
  (giving rise to the Gromoll filtration)
  and the alpha-invariant.  

  These elements are constructed by means of Morlet's homotopy equivalence
  $\Diff(D^6,\del)\simeq \Omega^{7}(PL_6/O_6)$ and Toda brackets in $PL_6/O_6$.
  %for the latter space. 
  We 
  %even
  {also} construct non-trivial elements of order $2$ in
  $\pi_{*}PL_m$ for every $m\ge 6$ and $*\equiv 1,2\pmod 8$ which are
  detected by the alpha-invariant.
  
  As consequences, we (a) obtain non-trivial elements of order $2$ in $\pi_{*}
  \Diff(D^m,\del)$ for $m\geq 6$ such that $*+m\equiv 0,1\pmod 8$; 
  (b) these elements
  remain non-trivial in $\pi_{*} \Diff(M)$ where $M$ is a closed spin manifold
  of the same dimension $m$ and $*>0$; 
  (c) they act non-trivially on the
  corresponding homotopy group of the space of metrics of positive scalar
  curvature of such an $M$; 
  in particular these homotopy groups are all non-trivial. The same applies to all other diffeomorphism
  invariant metrics of positive curvature, like the space of metrics of
  positive sectional curvature, or the space of metrics of positive Ricci
  curvature, provided they are non-empty.
  
  Further consequences are: (d) any closed spin manifold of dimension
  $m\geq 6$ admits a metric with harmonic spinors; 
  (e) there is no analogue of the odd-primary splitting of $(PL/O)_{(p)}$ 
  for the prime $2$;
  %\remdc{the splitting of $(PL/O)_{(p)}$ at odd primes $p$
  %for the prime $2$;}
   %May et al's splitting of $(PL/O)_{(p)}$ at the prime $2$; 
  (f) for any $bP_{8j+4}$-sphere $(j\geq 1)$ of order which divides $4$, the corresponding element in $\pi_0  \Diff(D^{8j+2},\del)$ lifts to $\pi_{8j-4}\Diff(D^6,\del)$,
  i.e.~lies correspondingly deep down in the Gromoll filtration.
\end{abstract}

%%%%%%%%%%%%%%%%%%%%%%%%%%%%%%%%%%%%%%%%%%%%%%%%

\section{Introduction} \label{sec:intro}
%%%%%%%%%%%%%%%%%%%%%%%%%%%%%%%%%%%%%%%%%%%%%%%%
We use the Gromoll filtration \cite{G} of 
$\Gamma^{n+1} = \pi_0(\Diff(D^n, \del))$ to study the topology of
spaces of metrics of positive curvature
%scalar 
and the topology of diffeomorphism groups 
for closed spin manifolds.

%Our basic theme is the Gromoll filtration \cite{G} of 
%$\Gamma^{n+1} = \pi_0(\Diff(D^n, \del))$ and applications to the topology of diffeomorphism
%groups and the space of metrics of positive 
%%scalar 
%curvature for closed spin
%manifolds.

This Gromoll filtration $\dots\subset\Gamma^{n+1}_{(k)}\subset
\Gamma^{n+1}_{(k+1)}\subset \dots \subset \Gamma^{n+1}_{(n)}=\Gamma^{n+1}$ is defined using the homomorphisms
\[  \lambda \colon \pi_{n-k}(\Diff(D^k, \del)) \to \pi_0(\Diff(D^n, \del)) = \Gamma^{n+1} \]
simply by setting
\[ \Gamma^{n+1}_{(k)} : = {\rm Im}(\lambda) \subset \Gamma^{n+1}.\]
Here $\lambda$ interprets a smooth family of diffeomorphisms of $D^k$ parametrized
by $D^{n-k}$ as one diffeomorphism of $D^{n}$ (which preserves the first $n{-}k$
coordinates). Our notation is somewhat non-standard, $\Gamma^{n+1}_{(k)}$ is
supposed to reflect the $k$-dimensional ``disk of origin'', as this is the
relevant parameter for our applications. (The more traditional notation for what we call $\Gamma^{n+1}_{(k)}$ is $\Gamma^{n+1}_{n-k-1}$.)

Our main result is that certain important classes in
$\Gamma^{n+1}$ have lifts all the way to
$\Gamma^{n+1}_{(6)}$. ``Important'' here refers in particular to classes which
have non-trivial $\alpha$-invariant, defined as follows and coinciding\footnote{as proved in
\cite[Section 3]{M}} with
Adams' invariant $d_{\reals}$ of \cite[Section 7]{Adams}. We consider the $\alpha$-invariant as a
homomorphism to real $K$-homology (of a point) 
\begin{equation*}
  \alpha_\Gamma\colon \Gamma^{n+1}\to KO_{n+1},
\end{equation*}
which factors in the following way
\begin{equation}\label{eq:alpha_factorized}
  \begin{CD}
 \alpha_\Gamma\colon   \Gamma^{n+1} @>{\iso}>{\Sigma}> \Theta_{n+1} @>>> \Omega_{n+1}^{\Spin} @>{\alpha_{\Spin}}>> KO_{n+1}.% \\
%   &&&& @AA{\iota}A\\
% && \pi_{n+1}^s @>{\iso}>{PT}> \Omega_{n+1}^{fr}
  \end{CD}
\end{equation}
Here, $\Theta_{n+1}$ is the group of oriented diffeomorphism classes of 
%exotic
\remdc{homotopy}
spheres, and the isomorphism $\Sigma$ produces an exotic
$(n+1)$-sphere from a diffeomorphism in $\Gamma^{n+1}$ by extending the latter
by the identity map to $S^{n}$ and then clutching two $(n{+}1)$-disks using this
diffeomorphism of $S^n$. The map to $\Omega^{\Spin}_{n+1}$ assigns to a
%an exotic 
\remdc{homotopy sphere}
the spin bordism class it represents (having a unique spin
structure). Finally, the transformation $\alpha_{\Spin}$ is
the so-called Atiyah orientation; it assigns to a spin manifold the KO-valued
index of its Dirac operator. % We include the stable homotopy groups of spheres,
% isomorphic to framed bordism by the Pontryagin-Thom construction $PT$, into the
% diagram because Adams' $d$-invariant is defined on stable homotopy, coincides
% with $\alpha_{\Spin}\circ\iota\circ PT$ and we are
% going to relate our constructions to his.

We will use many different versions of ``$\alpha$-invariant'' homomorphisms,
defined on different spaces. In most cases, we will not distinguish them in
notation but rather just write $\alpha$, the precise setting will be clear
from the context.

Recall that $\Gamma^{n+1}$ is a finite abelian group for each $n$, and
$KO_{n+1}=\integers/2\integers$ if $n\equiv 0,1$ modulo $8$, but is zero or
infinite cyclic for all other degrees. Therefore, $\alpha_\Gamma$ is
only interesting for $n\equiv 0,1\pmod 8$. It is a well known result of Adams
\cite[Section 7 and 12]{Adams} that $\alpha$ is a split 
epimorphism in these cases (if $n>0$). Our main result improves this by
constructing some elements with non-trivial $\alpha$-invariant deep in the
Gromoll filtration:

\begin{theorem} \label{thm:1}
For all $j \geq 1$ and $\epsilon \in \{1, 2\}$, 
there is a homotopy $(8j+\epsilon)$-sphere with disk of origin not bigger than $6$ and non-trivial $\alpha$-invariant, which is of order two in the group of homotopy spheres. 
% the $\alpha$-invariant is a split surjection already restricted the Gromoll
% filtration step $\Gamma^{8j+\epsilon}_{(6)}$ living on the $6$-sphere. 
In fact, somewhat more is true, namely
\[ \alpha \colon \pi_{8j-7+\epsilon}(\Diff(D^6, \del)) 
\to \Gamma^{8j+\epsilon}_{(6)}
\to KO_{8j+\epsilon} \]
is split surjective. 
\end{theorem}

In \cite{C-S} it was proven that $\alpha(\Gamma^{8j+2}_{(7)}) = KO_{8j+2}$.
In this paper we improve this result in two ways: we reduce the disk of origin
by one to $D^6$, and we also cover the dimensions $8j{+}1$.

To our knowledge, lifts this far in the Gromoll filtration have rarely been
constructed before. In addition, our construction methods seems to be
novel. In \cite{C-S}, 
%we 
\remdc{the first two authors}
constructed the required elements in
%$\Gamma^{n+1}_{(7)}$ 
\remdc{$\Gamma^{8j+2}_{(7)}$}
as products  between elements in $\pi_\beta(\Diff(D^k,\del))$
and $\pi_\alpha(S^\beta)$, a strategy which had been employed previously
by Antonelli, Burghelea, and Kahn \cite{A-B-K} and Burghelea and Lashoff \cite{B-L}.

In the present paper, we use a secondary product construction,
more precisely, Toda brackets.  In this way we implement the suggestion made in
\cite[Remark 2.15]{C-S}. As a further application of the method, 
%we also prove the following 
%theorem.
%
\remdc{we prove Theorem \ref{thm:bP_spheres} below.}
\remdc{
Let $\Gamma^{4i-1}_{bP} : = \Sigma^{-1}(bP_{4i})$ be the subgroup of $\Gamma^{4i-1}$
corresponding to those homotopy spheres which bound parallelizable manifolds.
%compact $4i$-dimensional manifolds. 
  %Recall that it has a unique subgroup of order $4$. 
Since $bP_{4i}$ is finite cyclic \cite{KM}, $\Gamma_{bP}^{4i-1} \cong bP_{4i}$ has a
unique subgroup of order $4$, which we denote by $_{4}\Gamma_{bP}^{4i-1}$.}

\begin{theorem}\label{thm:bP_spheres}
\remdc{
For all $j \ge 1$, every element of $_{4}\Gamma_{bP}^{8j+3}$
lies in $\Gamma^{8j+3}_{(6)}$. }

%For all $j\ge 1$ every element in this subgroup of $\Sigma^{-1}(bP_{8j+4})$
%lifts to $\Gamma^{8j+3}_{(6)}$. 
\end{theorem}

\noindent
For a summary of earlier results on the Gromoll filtration of $bP_{4k}$-homotopy 
spheres, see the bottom of the table in the Appendix \ref{sec:Gromoll}.

%%%%%%%%%%%%%%%%%%%%%%%%%%%%%%%%%%%%%%%%%%%%%%%%

\subsection{Harmonic spinors and diffeomorphism groups}
\label{sec:harmonic}

It is an old question whether a given closed spin manifold $M$ 
admits \emph{harmonic spinors}. Note that this depends on the Riemannian metric
$M$, the more precise question therefore is whether $M$ admits a Riemannian
metric such that its Dirac operator has non-trivial kernel.

This question has a long history. The many positive results all use the
following strategy: if every metric admits a harmonic spinor, we are of course
done. Otherwise, we look at the complement:
\begin{definition}
  Define $\Inv(M)$ to be the space of Riemannian metrics on $M$ with
  invertible Dirac operator. 
\end{definition}
It then suffices to show that this space is not contractible, so
that it can not be equal to the (contractible) space of \emph{all} Riemannian
metrics.

Nigel Hitchin \cite[Theorem 4.5]{H} was the first to use essentially
this method to prove that there are metrics with non-trivial harmonic spinor
whenever $\dim(M)\equiv -1,0,1\pmod 8$. 
Later, Christian B\"ar \cite{Baer} showed that the space of metrics with
non-invertible Dirac operator on any spin manifold of dimension $m\equiv 3\pmod 4$
is non-empty. Waterstraat \cite{Waterstraat} showed that
its components can be distinguished using the spectral flow of the Dirac
operator, which actually is a relative index.

More specifically, we assume that there is a metric $g_0\in\Inv(M)$.
Choose an embedding of $D^n$ into
$M$ and define $j\colon \Diff(D^n,\del)\to \Diff(M)$ via extension of a
diffeomorphism outside this embedded disk by the identity. 
We have the action map 
\begin{equation*}
  \Diff(M)\to \Inv(M); \quad f\mapsto f^*g_0,
\end{equation*}
given by pulling back $g_0$ by the diffeomorphism, which
we may compose with the extension map $j$.

Our goal now is to use this sequence of maps to obtain non-trivial elements in
$\pi_{n-m}(\Inv(M), g_0)$. Indeed, we can use a relative index of the Dirac
operator (the index difference to $g_0$ in the sense of Ebert \cite{Ebert})
\begin{equation*}
  \relind\colon \pi_{n-m}(\Inv(M),g_0)\to KO_{n+1}.
\end{equation*}
 Strictly speaking, in \cite{Ebert} the map is defined on the space of metrics
 of positive scalar curvature. However, the analytic
   condition required to construct it is \emph{not} positive scalar curvature
   but merily the 
   invertibility of the Dirac operator so that \cite{Ebert} literally applies.
   
The composition
\begin{multline}\label{eq:Hitch_seq}
  \pi_{n-m}(\Diff(D^m,\del))\to \pi_{n-m}(\Diff(M))\\
  \to
  \pi_{n-m}(\Inv(M),g_0) \xrightarrow{\relind} KO_{n+1}
\end{multline}
was introduced and studied by Hitchin \cite{H}. He proved that it
is equal to the $\alpha$-invariant homomorphism.

With Theorem \ref{thm:1} above we produce the required input for Hitchin's method to work in
almost all dimensions, therefore answering the question almost completely: 
\begin{theorem}\label{theo:harmonic}
  Let $M$ be a closed spin manifold of dimension $m\ge 6$. Then $M$ admits a
  Riemannian metric with a non-trivial harmonic spinor. \remws{Indeed for each  Riemannian metric $g_0$ in the complementary space $\Inv(M)$, the homotopy groups $\pi_{n-m}(\Inv(M), g_0)$ are non-trivial for $n\ge m$ and $n\equiv 0,1\pmod 8$.}
%   The complementary
%   space $\Inv(M)$ of metrics with invertible Dirac operator has non-trivial
%   homotopy 
%   goups in degree $n-m$, where $n\ge m$ and $n\equiv 0,1\pmod 8$.
\end{theorem}

\remws{Note that here $\Inv(M)$ is allowed to be empty, in which case the second statement is vacuous.}

\begin{proof}
%    Indeed, this is a direct consequence of \cite[Proposition 4.3]{H} and
%    \cite[Section 4]{H}. One just has to generalize the proof of \cite[Theorem
%    4.5]{H} 
%    which states that a closed spin manifold of dimension $m\equiv -1,0,1\pmod
%    8$ admits a metric with a non-trivial harmonic spinor by using the new
%    information about the Gromoll filtration of exotic spheres of Theorem
%    \ref{thm:1}.

%    Alternatively, one can directly observe that the index difference
%    homomorphism $\relind\colon \pi_{n-m}(\Pos(M),g_0)\to KO_{n+1}$ 
%    extends to
%    \begin{equation*}
% \relind\colon \pi_{n-m}(\Inv(M),g_0)\to KO_{m+1}.
% \end{equation*}
%  This follows because the analytic
%    condition required to construct it is \emph{not} positive scalar curvature
%    but merily the 
%    invertibility of the Dirac operator. As in Corollary \ref{corol:Diff_Pos},
%    if $g_0\in \Inv(M)$ we have
%    the action map $\Diff(D^m,\boundary)\to \Inv(M)$ which induces
%    \begin{equation*}
%      \pi_{n-m}(\Diff(D^m,\boundary))\to
%      \pi_{n-m}(\Inv(M))\xrightarrow{\relind}KO_{m+1}.
%    \end{equation*}
%
%We have already done almost all the work: The non-trivial classes of
%   order $2$ in $\pi_{n-m}(\Diff(D^m,\boundary))$ of Theorem
%   \ref{thm:1} which are
%   detected by $\alpha$ provide non-trivial classes in
%   $\pi_{n-m}(\Inv(M),g_0)$ which is placed in the middle of the
%   sequence~\eqref{eq:Hitch_seq}, so the assertion follows. 
%  
  \remws{We start by proving the  second assertion.  The non-trivial classes of
  order $2$ in $\pi_{n-m}(\Diff(D^m,\boundary))$ of Theorem
  \ref{thm:1} which are
  detected by $\alpha$, map to classes in
  $\pi_{n-m}(\Inv(M),g_0)$ through the action homomorphism; the  latter group is placed in the middle of the
  sequence~\eqref{eq:Hitch_seq}, so the classes constructed in this way are non-trivial.  
  
  It follows that $\Inv(M)$ is non-contractible (maybe empty) and therefore
  must be a strict subset of the contractible  space of all Riemannian metrics on $M$, and the first assertion follows.}
\end{proof}

\begin{remark}
  Bernd Ammann informs us that Theorem \ref{theo:harmonic} also follows as a
  special case of work he carried out independently and in parallel together
  with Bunke, Pilca, and Nowaczyk.  
  \remdc{This work has not appeared yet in preprint form.}
\end{remark}

\remdc{When $\Inv(M)\neq\emptyset$}
our proof gives a bit more than stated in Theorem \ref{theo:harmonic}:

\begin{corollary}\label{corol:harmonic_split}
%\remdc{Let $M$ be a closed spin manifold of dimension $m\ge 6$.  If $\Inv(M)\neq\emptyset$ then}
Under the assumptions of Theorem \ref{theo:harmonic}\remts{, and if $\Inv(M)\neq\emptyset$},
  \begin{equation*}
\pi_{n-m}(\Diff(M),\id) \to KO_{n+1}%=\integers/2
  \quad\text{ and}\quad \pi_{n-m}(\Inv(M), g_0) \to KO_{n+1} %=\integers/2
\end{equation*}
are split epimorphisms for all $g_0\in \Inv(M)$. This provides infinitely many degrees where the homotopy
groups contain a summand isomorphic to $\integers/2$.
\end{corollary}

\begin{remark}
Note that $\Inv(M)$ is non-empty if and only if the necessary condition for
  this is satisfied, namely that $\alpha(M)=0\in KO_m$, compare
  \cite{AmmannDahlHumbert_adv}.
\end{remark}

\begin{remark}
\remws{In the situation of Corollary \ref{corol:harmonic_split}, suppose that the hypothesis $\Inv(M)\neq\emptyset$ is omitted. Then our method still shows the existence of a split surjection $\pi_{n-m}\Diff(M) \to KO_{n+1}$, under the stronger hypothesis $n\geq m+2$, or after replacing $\Diff(M)$ by the ``spin diffeomorphism group'' whose elements are diffeomorphisms together with a lift of the derivative to the spin principal bundle. In this case the map to $KO$-theory is given by the $\alpha$-invariant of the mapping torus.}
%\remts{We leave it to the reader to the non-triviality result on $\pi_*(\Diff(M))$ to
%closed spin manifolds $M$ with $\Inv(M)=\emptyset$.}
\end{remark}
%
%%%%%%%%%%%%%%%%%%%%%%%%%%%%%%%%%%%%%%%%%%

\subsection{Positive curvature} \label{subsec:psc} 
%%%%%%%%%%%%%%%%%%%%%%%%%%%%%%%%%%%%%%%%%%%%%%%%%%%%%%%%

An important application of Theorem \ref{thm:1} concerns the topology of spaces
$\Pos(M)$ 
of metrics of suitable positive curvature on a closed spin manifold $M$ of
dimension $m$.
Here, $\Pos(M)$ can stand for any non-empty diffeomorphism invariant space of Riemannian
metrics which is contained in $\Inv(M)$.
% Let $\Pos(M)$ be
% the space of Riemannian metrics on $M$ with positive scalar curvature.
By the Schr\"odinger-Lichnerowicz formula this is the case for the space 
$\Riem^{+}_{sc}(M)$
of metrics of \emph{positive scalar curvature} on $M$. 
%Therefore, the same holds for
We list the most studied examples of $\Pos(M)$:
\begin{itemize}
\item the space $\Riem^{+}_{sc}(M)$ of positive scalar curvature metrics,
\item the space $\Riem^{+}_{Ric}$ of positive Ricci curvature metrics,
\item the space $\Riem^{+}_{sec}$ of positive sectional curvature metrics,
\item the space of $k$-positive Ricci curvature metrics for any $1\le k\le \dim(M)$, interpolating
  between the first two cases.
\end{itemize}

We are studying the case where the corresponding space $\Pos(M)$ is non-empty.
The Schr\"odinger-Lichnerowicz formula entails that the first obstruction to
the existence of a positive scalar curvature metric on $M$ is the index of the
Dirac operator defined by its spin structure, i.e., $\alpha_{\Spin}([M])$ of \eqref{eq:alpha_factorized}.
When $M$ is simply connected of dimension $\ge 5$, Stolz \cite{Stolz} proved
that $\Pos(M) \neq 
\emptyset$ if and only if $\alpha(M) = 0$.  In general, the question of whether
$\Pos(X) \neq \emptyset$ is a deep problem which remains open, see 
\cites{Rosenberg, Sch,ICM}. 

We start at the other end and we \emph{assume} that there is $g_0\in\Pos(M)$,
with $\Pos(M)$ as above.
As above, we have the embedding $j\colon \Diff(D^m,\del)\to \Diff(M)$ and the
action map 
\begin{equation*}
  \Diff(M)\to \Pos(M); \quad f\mapsto f^*g_0.
\end{equation*}

Note that the map $\Diff(M)\to \Inv(M)$ of Section \ref{sec:harmonic} factors
through this action map by the assumption $\Pos(M)\subset \Inv(M)$. Corollary
\ref{corol:harmonic_split} therefore gives immediately the following corollary.
\begin{corollary}\label{corol:Diff_Pos}
  Let $M$ be a closed spin manifold of dimension $m\ge 6$ with a Riemannian
  metric $g_0\in \Pos(M)$ for a space of metrics $\Pos(M)$ as above. If $n\equiv 0,1\pmod
  8$ and $n\ge m$, 
  then $KO_{n+1} = \Z/2$ and the composition
  \begin{multline*}
    \pi_{n-m}(\Diff(D^k,\del),\id) \to \pi_{n-m}(\Diff(M),\id)\\
    \to
    \pi_{n-m}(\Pos(M),g_0)\to \pi_{n-m}(\Inv(M),g_0) \to KO_{n+1} % = \integers/2\integers
  \end{multline*}
  is a split epimorphism. In particular, also
$\pi_{n-m}(\Pos(M), g_0) \to KO_{n+1}=\integers/2$ is a split epimorphism and
$\Pos(M, g_0)$
%, if non-empty, 
has infinitely many non-trivial homotopy groups.
\end{corollary}

\begin{remark}
Hitchin introduced precisely this method, applied to the space of
  metrics of positive scalar curvature in \cite{H}. However, at the time
  it was only known that
\begin{equation*}
\alpha\colon
\pi_{k}(\Diff(D^m,\del))\to KO_{m+k+1}
\end{equation*}
is surjective for $k=0$ or $k=1$,
and $m+k\equiv 0,1\pmod 8$, $m\ge 8$. Therefore, Hitchin with this method
only could obtain information about $\pi_0(\Riem^{+}_{sc}(M))$ and $\pi_1(\Riem^{+}_{sc}(M))$.
\end{remark}

Botvinnik, Ebert and Randal-Williams in the breakthrough paper \cite{BotvinnikEbertRandal-Williams}  study the space
of metrics of positive scalar curvature $\Riem^{+}_{sc}(M)$.
They show  that $\relind\colon
  \pi_{n-m}(\Riem^{+}_{sc}(M),g_0)\to KO_{n+1}$ is an epimorphism if $n\equiv 0,1\pmod
  8$ and has infinite image if $KO_{n+1}\iso\integers$, i.e.~$n\equiv 3\pmod
  4$.  Their methods are 
  rather different from ours, in particular the family of metrics they obtain
  are very inexplicit and rely on surgery.

  Hitchin's method, on the other hand, gives rather explicit families of
  metrics ---at least if the family of diffeomorphisms used in the
  construction is explicit. We view this as one of the appealing features of
  our construction.  Moreover, our method applies not only to scalar curvature,
  but to all metrics of positive curvature as listed above.

Note that in Hitchin's and therefore our construction of homotopy classes of
metrics of positive scalar curvature, the corresponding families of metrics are obtained by
pulling back $g_0$ with an appropriate family of diffeomorphisms which is
supported on 
a small disk in $M$. This means that we only make a local
change of the given initial metric $g_0$. We note that
by the very way they are constructed these classes become trivial when mapped to the moduli
space of metrics (in contrast to some elements of $\pi_*(\Riem^{+}_{sc}(M))$ obtained in very different ways in \cite{BotvinnikEbertRandal-Williams,HSS}).

\subsection{Toda brackets} \label{subsec:Toda_intro}
%%%%%%%%%%%%%%%%%%%%%%%%%%%%%%%%%%%%%%%%%%
We now describe in more detail our method to prove the main Theorem
\ref{thm:1}, lifting certain exotic spheres deep in the Gromoll filtration,
and additional results around this. 

The starting point of the construction is a homotopy equivalence 
%(for $n\ne 4$) \comws{Where do we need this assumption?}
\begin{equation*} \label{eq:Morlet}
  M\colon \Diff(D^n,\boundary)\to \Omega^{n+1}(PL_n/O_n)
\end{equation*}
due to Morlet \cite{Morlet}, with a detailed proof by Burghelea and Lashof in \cite[Theorem
4.4]{B-L}. {Recall that $PL_n$ is the simplicial group of piecewise linear homeomorphisms of
$\reals^n$ fixing the origin, with homotopy theoretic subgroup inclusion $O_n
\to PL_n$ for the orthogonal group
$O_n$.} One sets $O:=\lim_{n\to\infty} O_n$, $PL:=\lim_{n\to\infty} PL_n$,
and $PL/O:=\lim_{n\to\infty}(PL_n/O_n)$.
There are of course stabilization maps
$PL_n/O_n\to PL_{n+1}/O_{n+1}\to PL/O$ (we call all these stabilization maps
$S$).  {We will also use the orientation preserving versions, denoted
  $SPL_n$, etc}.

As checked in \cite[Theorem 1.3]{B-L} and \cite[Lemma 2.5]{C-S}, under the
isomorphism induced by $M$, the stabilization $\lambda$ defining
the Gromoll filtration becomes the stabilization $S$, i.e.~we have a
commutative diagram
\begin{equation*}
  \begin{CD}
  \pi_{n-k}(\Diff(D^{k},\partial))
  @>{\lambda}>>\pi_{n-k-1}(\Diff(D^{k+1},\partial)) @>{\lambda}>>  \pi_0(\Diff(D^n,\partial))\\% @= \Gamma_{n+1}\\
    @V{\iso}V{M_*}V @V{\iso}V{M_*}V @V{\iso}V{M_*}V\\ % @V{\iso}V{M_*}V\\
  \pi_{n+1}(PL_k/O_k) @>{S_*}>> \pi_{n+1}(PL_{k+1}/O_{k+1}) @>{S_*}>>
  \pi_{n+1}(PL_n/O_n), %@>{S_*}>{\iso}>  \pi_{n+1}(PL/O)
  \end{CD}
\end{equation*}
where the group in the bottom right corner is already stable,
i.e.~the stabilization map to $S_* \colon \pi_{n+1}(PL_n/O_n) \to \pi_{n+1}(PL/O)$ 
is an isomorphism \cite[Theorem 4.6]{B-L}.
Indeed, as verified in \cite[\S 2]{C-S},
the fundamental theorem of smoothing theory \cite[Theorem 7.3]{Lance}
gives an isomorphism
\[ \Psi \colon \Theta_{n+1} \xrightarrow{\cong} \pi_{n+1}(PL/O) \]
such that $S_* \circ M_* = \Psi \circ \Sigma \colon \Gamma^{n+1} \xra{~\cong~}\pi_{n+1}(PL/O)$.

It follows that finding elements deep in the Gromoll filtration corresponds
to lifting elements of $\pi_{n+1}(PL/O)$ to $\pi_{n+1}(PL_k/O_k)$.
In the predecessor paper \cite{C-S} this was achieved by using compositions
\[ \pi_{i}(S^j) \times \pi_j(PL_k/O_k)  \to \pi_{i}(PL_k/O_k).  \]

In this paper we show how Toda brackets on elements of $\pi_*(PL_k/O_k)$
can be used to go even deeper in the Gromoll filtration. Recall that the Toda
bracket $\an{f,g,h}$ of three homotopy classes of maps $f\colon X\to Y$,
$g\colon Y\to Z$, $h\colon Z\to W$ is defined (as a secondary composition
product) whenever the compositions of two
consecutive maps are homotopic to constant maps. The Toda bracket is a set of homotopy
classes of maps from $\Sigma X$ to $W$. The indeterminacy depends on the null
homotopies one can choose.

To prove Theorem \ref{thm:1}, we start with the unique element of order two,
%
%\[a_{PL_6/O_6} \in \pi_7(PL_6/O_6) \overset{M^{-1}_*}{\cong\,} \Gamma^7 
%\overset{\Sigma}{\cong} \Theta_7 \cong \Z/28,\]
%
  \begin{equation}\label{eq:a6}
a_{PL_6/O_6} \in \pi_7(PL_6/O_6) \cong \Theta_7
 \cong \Z/28,
 \end{equation}
where the first isomorphism is $\Psi_*^{-1} \circ S_*$ and the second isomorphism
is found in \cite{KM}.
%where the last isomorphism is a well-known 
%computation from the surgery exact sequence (see e.g.\ \cite[Remark 6.51]{Lueck}). 
For $\epsilon\in \{1,2\}$, we now form Toda brackets 
%of the form 
\[\an{\mu_{8j-8+\epsilon,7}, 2_7, \aelt_{PL_6/O_6}}\subset \pi_{8j+\epsilon}(PL_6/O_6)\]
with certain elements
\[\mu_{8j-8+\epsilon,7}\in \pi_{8j-1+\epsilon}(S^7) \quad (j\geq 1)\]
of order 2. 
These elements were constructed in \cite[Section 2.3]{C-S} and stabilize to a family of elements in $\pi_{8j-8+\epsilon}^s$ constructed by Adams \cite{Adams}.
%considered by Adams.

The main task is the to compute the $\alpha$-invariant of an element in this Toda bracket. As an ingredient, we use the fact that the $\alpha$-invariant from \eqref{eq:alpha_factorized} is induced by a map of spaces
\[ \alpha_{PL/O} \colon PL/O\to \Omega^\infty \mathbf{KO}\]
so that it is natural with respect to composition products and Toda brackets. We obtain:

\begin{theorem} \label{thm:PL/O} 
%\hfill
\begin{enumerate}
\item[(a)] For all $j \geq 1$ and $\epsilon \in \{1, 2\}$, any element in the Toda bracket 
%$\an{\mu_{8j-8+\epsilon,7}, 2_7, \aelt_{PL_6/O_6}}$ 
$$\an{\mu_{8j-8+\epsilon,7}, 2_7, \aelt_{PL_6/O_6}} \subset \pi_{8j+\epsilon}(PL_6/O_6)$$ 
has non-trivial $\alpha$-invariant.
\item[(b)] There is an element of order 2 in this Toda bracket.
\end{enumerate}
\end{theorem}

%Theorem \ref{thm:1} follows from Theorem \ref{thm:PL/O} by translating the elements from 
%(b) back through the  Morlet equivalence of \eqref{eq:Morlet}.

%\begin{proof}[Proof of Theorem \ref{thm:1}]
%The theorem follows from Theorem \ref{thm:PL/O} by translating the elements from 
%(b) back through the  Morlet equivalence of \eqref{eq:Morlet}. \remts{Explicitly,
%let $\beta\in \pi_{8j+\epsilon}(PL_6/O_6)$ be one of the elements of order $2$
%whose existence is asserted by the Theorem. Via the Morlet isomorphism this
%gives rise to $M_*^{-1}(b)\in \pi_{8j-7+\epsilon}(\Diff(D^6,\partial))$ of
%order precisely $2$ and with non-trivial $\alpha$-invariant from
%\eqref{eq:alpha_factorized}, which we show to be induced from $\alpha_{PL/O}$
%and the Morlet equivalence.}
%\end{proof}

\begin{proof}[Proof of Theorem \ref{thm:1}]
The theorem follows from Theorem \ref{thm:PL/O} by translating the elements from 
(b) back through the  Morlet equivalence of \eqref{eq:Morlet}. \remts{Explicitly,
let $b \in \pi_{8j+\epsilon}(PL_6/O_6)$ be an element of order $2$
as in Theorem \ref{thm:PL/O}(a). Via the Morlet isomorphism, 
we obtain $M_*^{-1}(b)\in \pi_{8j-7+\epsilon}(\Diff(D^6,\partial))$ of
order $2$ with non-trivial $\alpha$-invariant and hence
$\alpha \colon \pi_{8j-7+\epsilon}(\Diff(D^6, \del)) 
\to KO_{8j+\epsilon}$ is split surjective.
}
\end{proof}
% 
% 
%%%%%%%%%%%%%%%%%%%%%%%%%%%%%%%%%%%%%%%%%%

\subsection{The space $PL/O$ and $\im(J)$-homotopy spheres}
%\subsection{Classifying spaces for smoothing theory}
%%%%%%%%%%%%%%%%%%%%%%%%%%%%%%%%%%%%%%%%%%

Recall $\pi_k^s := \colim_{i \to \infty}\pi_{i+k}(S^i)$ % be the stable $k$-stem 
and the $J$-homomorphism $\Jhom_* \colon \pi_*(O) \to \pi_*^s$.
The Kervaire-Milnor exact sequence \cite{KM},
\begin{equation} \label{eq:KMS}
0 \to bP_{n+2} \to \Theta_{n+1} \xra{~\Phi~} \coker(\Jhom_{n+1}),
\end{equation}
is an exact sequence of finite abelian groups which 
is split short exact at any odd prime $p$ by a result of Brumfiel \cite[Theorem 1.3]{Brumfiel1}.
Actually, there is a canonical splitting of the $p$-localization
\begin{equation} \label{eq:PL/O_split}
(PL/O)_{(p)} \sim N_{(p)}\times C_{(p)},
\end{equation}
with isomorphisms 
$\pi_*(N_{(p)}) \iso (bP_{*+1})_{(p)}$ and $\pi_*(C_{(p)})\iso \coker(\Jhom_{*})_{(p)}$ 
which, via the isomorphism $\Psi \colon \Theta_{*} \iso \pi_*(PL/O)$,
induce Brumfiel's splitting of $(\Theta_{*})_{(p)}$; see \cite[Theorem 6.8 (iii)]{May} and \cite{May2}\footnote{
These results there are stated for the space $(TOP/O)_{(p)}$ but  
$(PL/O)_{(p)} \simeq (TOP/O)_{(p)}$ at odd primes by
Kirby-Siebenmann \cite[V Theorem 5.3]{KS}.}.
Here we have written $A_{(p)}$ for the localization of an abelian group $A$.
\remdc{The following theorem shows that such a splitting cannot exist 
at the prime $p=2$.}
%%%%%%%%%%

%%%%%%%%%

% $\phi_{(p)} \colon A_{(p)} \to B_{(p)}$ for the $p$-localisation
% of a homomorphism $\phi$.

\begin{theorem} \label{thm:PL/O_at_2}
%%%%%%%%%%%%%%%%%%%
%At the prime $2$, a corresponding spliting can not exist, in particular
%there is no space $N_{(2)}$ with map $N_{(2)}\to (PL/O)_{(2)}$ implementing
%the inclusion of $(bP_{*+1})_{(2)}\subset (\Theta_{*})_{(2)}=\pi_{*}((PL/O)_{(2)})$.
%
\remdc{
There is no space $N_{(2)}$ with map $t \colon N_{(2)}\to (PL/O)_{(2)}$ 
such that $t_* \colon \pi_*(N_{(2)}) \to \pi_*((PL/O)_{(2)})$ is 
injective with image $\Psi^{-1}\bigl( (bP_{*+1})_{(2)} \bigr)$.}
\end{theorem}

\begin{proof}
The exotic sphere corresponding to the element $a_{PL_6/O_6}\in \pi_7
(PL_6/O_6)$ from \eqref{eq:a6} is a $bP$-sphere. So it would define an element
\remdc{$t_*^{-1}(a_N) \in \pi_7( N_{(2)})$} if such a space \remdc{and map} existed. 
Then, any element in the
Toda bracket
\remdc{$\an{\eta_7, 2_7, t_*^{-1}(a_N)}\subset \pi_{9}(N_{(2)})$} would 
\remdc{map under $\Psi \circ t_*$,}
%correspond to,
  perhaps up to odd multiples, a $bP$-sphere in dimension 9. But {any
  $bP$-sphere}
has trivial $\alpha$-invariant, contradicting Theorem \ref{thm:PL/O}.
\end{proof}

Let $G_n$ be the topological monoid of 
self-homotopy equivalences of the $(n{-}1)$-sphere:
\[ G_n := \{\phi \colon S^{n-1} \xrightarrow{\simeq} S^{n-1} \}; 
\qquad G=\lim_{n\to \infty} G_n,\]
and let $SG_n$ and $SG$ %\remdc{be the submonoids} 
be their orientation preserving variants, 
consisting of maps of degree $1$.
To consider the situation when $p=2$ we recall that
\cite[V \S4]{May} defined an equivalence
\begin{equation} \label{eq:SS}
\psi \colon SG \simeq \JS_\infty \times C_\infty,
\end{equation}
where $\JS_\infty: = {\prod_{p}%\substack{\bigtimes\\p}\,
  \JS_p}$ for certain $p$-local spaces $\JS_p$,
  $C_\infty : = {\prod_p%\substack{\bigtimes \\p}\,
  C_p}$
and we have $\pi_*(\JS_\infty) \cong \im(\Jhom_*) \oplus {\rm Tors}(KO_*)$ and 
$\pi_*(C) \cong \coker(\Jhom_*)/{\rm Tors}(KO_*)$.  
In Section \ref{sec:ImJ} we show that
the splitting \eqref{eq:SS} gives rise to a splitting of the $\alpha$-invariant
\[ s_* \colon {\rm Tors}(KO_{n+1}) \to \coker(\Jhom_{n+1}).\]
Using the Kervaire-Milnor homomorphism $\Phi \colon \Theta_{n+1} \to \coker(\Jhom_{n+1})$,
we say that a homotopy $(n{+}1)$-sphere
$\Sigma$ is an {\em $\im(J)$-sphere} if
\[ \Phi(\Sigma) \in s_*({\rm Tors}(KO_{n+1})) \]
and we define 
$\Theta^J_{n+1} \subset \Theta_{n+1}$
to be the subgroup of $\im(J)$-homotopy spheres.  
{We compute in Lemma \ref{lem:ImJ_and_Brumfiel}}%
\[ \Theta^J_{n+1} \cong bP_{n+2} \oplus {\rm Tors}(KO_{n+1})\]
with 
%$\alpha|_{\Theta_{n+1}^J} = 0 \oplus \id$.
$\alpha(bP_{n+1}) = \{0\}$ and $\alpha(\Theta^J_{n}) = {\rm Tors}(KO_{n})$.

{We show that Theorem \ref{thm:1} can be made more explicit, regarding
  $\im(J)$-spheres. This relies on an $\im(J)$-version of Theorem
  \ref{thm:PL/O}, which in turn uses $\im(J)$-versions $PL_n^J$ of $PL_n$ etc.}

\begin{theorem} \label{thm:ImJ_spheres}
%%%%%%%%%%%%%%%%%%%%
For all $j \geq 1$ and $\epsilon \in \{1, 2\}$, 
there is an $\im(J)$-homotopy sphere
{$\Sigma\in \Theta^J_{8j+\epsilon}$} of order two with $\alpha(\Sigma) = 1$ and disk origin at most $6$.
%;
%i.e.~$\Sigma^{-1}(\Sigma) \in \Gamma^{8j+\epsilon}_{(6)}$.
%
\end{theorem}

%
%%%%%%%%%%%%%%%%%%%%%%%%%%%%%%%%%%%%%%%%%%

\subsection{Some new elements of $\pi_*(PL_m)$}
%%%%%%%%%%%%%%%%%%%%%%%%%%%%%%%%%%%%%%%%%%
Above, our interest in the space $PL_6/O_6$ arose due to the Morlet equivalence
$\Diff(D^6, \partial) \simeq \Omega^7(PL_6/O_6)$.  But 
the space $PL_6/O_6$, and more generally the spaces 
$PL_m/O_m$, $PL_m$ and $TOP_m$, $TOP_m/O_m$,
have an important role in smoothing theory and a long history of study in their own right.
Here $TOP_m$ is the space base of base-point preserving homeomorphisms
of $\R^m$ and $TOP_m/O_m$ is its quotient by the orthogonal group, with
$TOP=\lim_m TOP_m$ and $TOP/O=\lim_m TOP_m/O_m$.

We  mention just one recent 
major breakthrough concerning the homotopy theory of the spaces above, based
on the fundamental work of Galatius and Randal-Williams \cite{G-RW}:
 Weiss \cite[Appendix B]{W4}  proves 
that the Pontrjagin class defines a non-trivial homomorphism
\[ \pi_{4k-1}(TOP_m) \to \Q \]
in a range of dimensions where the corresponding homomorphism on $\pi_{4k-1}(O_m)$ vanishes.
Hence Weiss shows the existence of classes in $\pi_{4k-1}(TOP_m) \otimes \Q$
which map non-trivially to $\pi_{4k-1}(TOP_m/O_m) \otimes \Q$.

Below we describe how computations with Toda brackets give new 
information about the $2$-primary homotopy structure of the spaces listed above,
and specifically about $2$-torsion in $\pi_*(PL_m)$ and $\pi_*(TOP_m)$.

Consider the following homotopy commutative diagram which gives
a space level description of the $\alpha$-invariant (see Section \ref{subsec:alpha}):
\[ \xymatrix{  
PL_6 \ar[d] \ar[r] &
PL_6/O_6 \ar[d] \ar[r] &
PL/O \ar[d] \ar@/^2.5ex/[drrr]^{\alpha_{PL/O}}
\\
TOP_6 \ar[r] &
TOP_6/O_6 \ar[r] &
TOP/O \ar[r] &
G/O \ar[r] &
\Omega^\infty \MSpin \ar[r] &
KO
}\]
%
% Here $TOP/O = \lim_{m \to \infty} TOP_m/O_m$ and $G/O$ 
% the homotopy fibre of the natural map $BO \to BG$ where $BG$
% is the classifying space for stable spherical fibrations.
For any space $X$ in the above 
%fibration,
\remdc{diagram}
we write $\alpha \colon \pi_*(X) \to KO_*$ for the map
induced on homotopy groups by the corresponding map $X \to KO$.
%in the diagram above.

The methods described in Section \ref{subsec:Toda_intro} required as input
the order two homotopy class $\aelt_{PL_6/O_6} \in 2 \pi_7(PL_6/O_6)$,
the subgroup of elements divisible by $2$.
In Section \ref{subsec:alpha_on_PL6} we show that $\aelt_{PL_6/O_6}$ lifts 
to an element of order two $\aelt_{PL_6} \in 2 \pi_7(PL_6)$ and this allows us to prove

\begin{theorem} \label{thm:PL6}
%%%%%%%%%%%%%%%%%
For all $j \geq 1$, $\epsilon \in \{1, 2\}$ and $m \geq 6$, 
the $\alpha$-invariant 
$$\alpha_{} \colon \pi_{8j+\epsilon}(PL_m) \to KO_{8j+\epsilon}$$
is a split surjection.  
Hence the same holds for $\alpha \colon \pi_{8j+\epsilon}(TOP_m) \to KO_{8j+\epsilon}$,
$\alpha \colon \pi_{8j+\epsilon}(PL_m/O_m) \to KO_{8j+\epsilon}$
and 
$\alpha \colon \pi_{8j+\epsilon}(TOP_m/O_m) \to KO_{8j+\epsilon}$.
\end{theorem}

The rest of this paper is organised as follows:
In Section \ref{sec:Hitchin}  we
establish basic facts about Toda brackets in the space $SG_n$,
mod~$2$ homotopy groups and the space level $\alpha$-invariant. 
Section 3 is about the $\alpha$-invariant on
$PL_6/O_6$ and $PL_6$ and Theorems 
%\ref{thm:1}?
\ref{thm:PL/O} and \ref{thm:PL6} are proven there.
Section 4 covers $\im(J)$-homotopy spheres and contains the proofs of Theorems
\ref{thm:ImJ_spheres} and \ref{thm:bP_spheres}.
%Section 5 is about diffeomorphism groups, containing the proof of 
%Theorem \ref{thm:diff}.

\medskip
{\bf Acknowledgements:} We would like to thank Peter May for helpful 
comments concerning the splitting of $(PL/O)_{(p)}$ in \eqref{eq:PL/O_split}.
\remts{We also thank the referree for many suggestions which have improved 
the presentation.}

%
% Introduction end
%%%%%%%%%%%%%%%%%%%%%%%%%%%%%%%%%%%%%%%%%%%%%%%%

\section{Toda brackets, $\pi_*^M(X)$ and the $\alpha$-invariant}
%\section{Toda brackets, mod~$2$ homotopy and the $\alpha$-invariant}
%\section{The Gromoll filtration of Hitchin spheres}
\label{sec:Hitchin}
%%%%%%%%%%%%%%%%%%%%%%%%%%%%%%%%%%%%%%%%%%%%%%%%
\newcommand{\Junst}{I}
\newcommand{\Jst}{I}
\newcommand{\Jsp}{I_{sp}}

% (fold)
\subsection{Toda brackets in $SG_n$} 
\label{sec:J-homomorphism}
%%%%%%%%%%%%%%%%%%%%%%%%%%%%%%%%%%%%%%%%%%%%%%%%

In this subsection we review the canonical homomorphism 
\[ \Junst \colon \pi_k(SG_n) \to \pi_{k+n}(S^n) \]
and Toda brackets.

Let $h \colon X \to SG_n$ be a map.
The adjoint of $h$ is given by
\[ \wh h \colon X\times S^{n-1} \to S^{n-1},
\quad (x, y) \mapsto h(x)(y) . \]
For any map $\psi \colon X \times Y \to Z$, the Hopf construction on $\psi$
\cite[Ch.~XI, \S 4]{Whitehead} is the map
\[ S \psi \colon X \ast Y \to \Sigma Z,
\quad [x, t, y] \mapsto [\psi(x, y), t], \] 
were $\ast$ denotes join and $\Sigma$ denotes suspension.  
For any space $X$ we identify $X \ast S^{n-1} = \Sigma^nX$.
Identifying further $\Sigma^n S^k$ with $S^{n+k}$ we obtain a map
\[ \Jsp \colon {\rm Map}(X, SG_n) \to 
{\rm Map}(\Sigma^n X, S^{n}),
\quad h \mapsto S \wh h.\]
%

% In the case where $X$ is a sphere we identify 
% $S^k \ast S^{n-1} = S^{k + n}$
% and $\Sigma S^{n-1} = S^n$ so that
%
%\[ S\wh h \colon S^{k + n} \to S^n
%,
%\quad [x, t, y] \mapsto [\wh h(x, y), t] = [h(x)(y), t].
%\]
%
%%%%%% DC's base-point issues %%%%%%
\begin{comment}
To pass to homotopy groups, lets take some care with basepoints.
It is standard to take the identity as the base-point in $SG_n$.
We let $w_0 \in S^k$, $[w_0, 0, x_0] \in S^{k+n} = S^k \ast S^{n-1}$ 
and $[x_0, 0] \in S^n = \Sigma S^{n-1}$ be the basepoints.
For a based map $f \colon (S^k, w_0) \to (SG_n, {\rm Id})$,
we have
%
\[ S\wh h \colon (S^{k+n}, [w_0, 0, x_0]) \to (S^n, [x_0, 0]). \]
%
\begin{remark} \label{rem:pre-Toda}
Notice then that for the constant map $h \colon S^k \to {\rm Id} \in SG_n$,
that $S \wh {\rm Id} \colon S^{k+n} \to S^n$ is not the constant map.
This will presumably be relevant later when we come to Toda brackets.
\end{remark}
%
Passing to based homotopy sets, denoted $[..., ...]_\ast$, $J_{\rm Map}$ induces a map
%
\[ J \colon [(S^k, w_0), (SG_n, {\rm Id})]_{\ast} 
\to [(S^{k+n}, [w_0, 0, x_0]), (S^n, [x_0, 0])]_{\ast}. \]
%
Since the left hand side is by definition $\pi_k(SG_n)$ and the right hand side
is $\pi_{k+n}(S^n)$, the above map can be written
%
\[ J \colon \pi_k(SG_n) \to \pi_{k+n}(S^n).   \]
%
\end{comment}
%
Passing to path-components we obtain a map
\[ \Jsp \colon [X, SG_n] \to [\Sigma^n X, S^n], \quad
[h] \mapsto [S \wh h] .\]
If we set $X=S^k$, we obtain the homomorphism
\[ \Junst \colon \pi_k(SG_n) \to \pi_{k+n}(S^n), \]
noting that we can identify unpointed with pointed homotopy classes, since $SG_n$ and $S^1$ are connected $H$-spaces and since $S^n$ is simply-connected for $n>1$.
If $0<k < n-1$, then $\pi_k(SG_n) \cong \pi_k(SG)$ and $\pi_{k+n}(S^n)\cong \pi_k^s$ are stable and $\Junst$ is an isomorphism 
\[ \Jst \colon \pi_k(SG) \xrightarrow{\cong} \pi_k^s. \]

\begin{remark}
%%%%%%%%%%%%%%
{The isomorphism $\Jst$ is often taken as an identification, e.g.~in \cite{C-S}.}%   However we introduce notation here,
% as it will be useful in the proofs which follow concerning Toda brackets.
%
\end{remark}

Recall that if 
 \[X\xra{f} Y \xra{g} Z\xra{h} W\]
is a sequence of continuous maps such that both composites $gf$ and $hg$ are nullhomotopic, then the Toda bracket
\[\an{f, g, h} \subset [\Sigma X, W]\]
is defined as the set of all homotopy classes constructed in the following
way: Choose nullhomotopies $H$ and $K$ of $gf$ and $hg$ so to obtain two
nullhomotopies $h_*H$ and $f^*K$ of the triple composite $hgf$, thus a map
{from $\Sigma X$ to $W$}. By construction, if $k\colon W\to W'$ is another map, then 
\[k_*\an{ f, g, h} \subset \an{ f, g, kh}\subset [\Sigma X, W'].\]

%We will be interested in the situation where $W=SG_n$. We have:

\begin{lemma} \label{lem:J_and_Toda}
%%%%%%%%%%%%%
Let $f \colon S^i \to S^j$, $g \colon S^j \to S^k$ and
$h \colon S^k \to SG_n$ be maps such that $\an{f, g, h}$ 
is defined.  Then
\[ \Junst \bigl( \an{f, g, h} \bigr) \subset 
\an{\Sigma^n f, \Sigma^n g, \Junst h} \subset 
\pi_{i+1+n}(S^n). \]
\end{lemma}

\begin{proof}
Since $\Jsp$ is natural in $X$, $\Junst h$ is given by suspending $h$ first $n$ times and then postcomposing with 
\[\varepsilon:= \Jsp (\id_{SG_n})\colon \Sigma^n (SG_n)  \to  S^n.\]
Applying this reasoning to $\an{f,g,h}$ instead of $h$ we have
\begin{multline*}
 \Junst \bigl(\an{f,g,h}\bigr) = \varepsilon_*\bigl(\Sigma^n \an{f,g,h}\bigr)\\
 \subset \varepsilon_*\bigl(\an{\Sigma^n f, \Sigma^n g, \Sigma^n h}\bigr) 
 \subset \an{\Sigma^n f, \Sigma^n g, \varepsilon\circ (\Sigma^n h)}
\end{multline*}
where the last map is $\Junst h$. 
\end{proof}

% subsection J-homomorphism (end)
%%%%%%%%%%%%%%%%%%%%%%%%%%%%%%%%%%%%%%%%%%%%%%%%

% (fold)
\subsection{Mod~$2$ homotopy groups} 
\label{subsec:Moore_space}
%%%%%%%%%%%%%%%%%%%%%%%%%%%%%%%%%

In this subsection we recall how maps of mod~$2$ Moore spaces
are related to certain Toda brackets. 
%Then we recall some
%basic facts about homotopy classes of maps of the mod~$2$ Moore space.

Let 
\[ M_k := S^k \cup_2 e^{k+1} \]
be the mod-$2$ Moore space,
and $c \colon M_k \to S^{k+1}$ the map collapsing $S^k$ to a point.
If $X$ is a simply-connected space and $x \colon S^k \to X$ is such that $2x = 0 \in \pi_k(X)$,
then $x$ can be extended to a map $\bar x \colon M_k \to X$.
%Suppose further that 
\remdc{Moreover, if} $y \colon S^{i-1} \to S^{k}$ is a map such that
$2y = 0 \in \pi_{i-1}(S^{k})$
then there is a map $\ol{Sy} \colon S^{i} \to M_k$ such that
$c \circ \ol{Sy} = Sy$.

Thus we can form the composition $\bar x \circ \ol{Sy} \colon S^i \to M_k \to X$,
and it follows from the definitions that
\[ \bar x \circ \ol{Sy} \in \an{y, 2, x} \subset \pi_{i}(X).  \]
%\[ \{ \bar x \circ \ol{Sy} \, | \,
%\bar x|_{S^k} = x, c \circ \ol{Sy} = Sy \} = \an{y, 2_k, x} \subset \pi_{i}(X).  \]
%
Here we identify $\pi_k(S^k)=\Z$ via the mapping degree. Moreover, the choices in the construction of $\ol{Sy}$ and $\bar x$
correspond precisely to the indeterminacy in the Toda bracket.

To apply the above, we will be interested in pointed homotopy classes 
of base-point preserving maps $\bar x \colon M_k \to X$.
Let $X$ be a simply connected space and define the $k$-th homotopy
set with $\Z/2$-coefficients, see \cite[\S 3]{N}, by
\[ \pi_k^M(X) := [M_k, X]_*.\]
Note that \cite{N} uses different notation, with
$\pi_k^M(X)=\pi_{k+1}(X;\integers/2)$.
Notice that for $k \geq 2$, $M_k$ is a suspension and so a 
co-H-space and so $\pi_k^M(X)$ has a \remdc{natural} group structure 
% given by 
% setting
% %
% \[ f+g := (f \vee g) \circ \nu \colon M_k \xra{~\nu~} M_k \vee M_k 
% \xra{f \vee g} X,  \]
% % 
% where $\nu \colon M_k \to M_k \vee M_k$
% is the pinch map, and $[f]+[g] = [f + g]$.
%[(f \vee g) \circ \nu]$,

The cofibration sequence $S^k\xrightarrow{i} M_k\xrightarrow{c} S^{k+1}$ comes with a long exact Puppe sequence \cite[\S 18]{N}
\begin{equation} \label{eq:les_of_moore_groups}
  \dots \to \pi_{k+1}(X) \xra{\times 2\,} \pi_{k+1}(X)
\xra{~c^*~} \pi_k^M(X) \xra{~i^*~} \pi_k(X) \xra{\times 2\,} \pi_k(X) \to \dots~. 	
\end{equation}
For an abelian group $A$, let $\eot{A} \subset A$ denote the subgroup of elements of order less
than or equal to two and let $A/2$ denote the quotient $A/2A$.
The long exact sequence above gives rise to a short exact sequence,
see \cite[\S 3]{N},
\begin{equation}\label{eq:ses_of_moore_groups}
  0 \to \pi_{k+1}(X)/2 \xrightarrow{c^*} \pi_k^M(X) \xrightarrow{i^*} \eot{\pi_k(X)} \to 0. 
\end{equation}

Denote by $\eta\colon S^{k+1}\to S^k$ the non-trivial homotopy class.
The following lemma is probably well known to experts, 
but as we did not find a reference, we give a proof.

\begin{lemma} \label{lem:Moore_extension}
%%%%%%%%%%%%%
Let $x \in \eot{\pi_k(X)}$ and let $\bar x \in \pi_k^M(X)$ have
$i^*(\bar x) = x$.  Then $2 \bar x = [x \circ \eta] \in \pi_{k+1}(X)/2$.
\end{lemma}

\begin{proof}
By naturality it is enough to consider the case where $X=M_k$ and $\bar x=1\in \pi_k^M(M_k)$, represented by the identity map. 
% Let $f \colon M_k \to X$ be given and note that
% %
% \[  2f = f+f = (f \vee f) \circ \nu
% = f \circ ({\rm Id} \vee {\rm Id}) \circ \nu. \]
% %
% Hence we consider the map $2 := {\rm Id \vee \rm Id} \circ \nu \colon M_k \to M_k$
% and $2 \in \pi_k^M(M_k)$.  
Now $\pi_k(M_k)=\Z/2[i]$ where 
$i \colon S^k \to M_k$ is the inclusion, and $\pi_{k+1}(M_k) = \Z/2[i \circ \eta]$. Hence the short exact sequence \eqref{eq:ses_of_moore_groups} becomes
\[  0\to \Z/2 \xra{~c^*~} \pi_k^M(M_k) \xra{~i^*~} \Z/2  \to 0,  \]
and we must show that $2 \in \pi_k^M(M_k)$
is equal to $c^*(i \circ \eta)$, i.e.~that the sequence does not split, or
equivalently that $\pi_k^M(M_k)\cong \integers/4$. This is exactly the
statement of \cite[Corollary 7.3]{N}.
\end{proof}

% subsection:Moore_space(end)
%%%%%%%%%%%%%%%%%%%%%%%%%%%%%%%%%

% (fold)
\subsection{The $\alpha$-invariant} 
\label{subsec:alpha}
%%%%%%%%%%%%%%%%%%%%%%%%%%%%%%%%%%%%%%%%%%%%%%%%

Recall that the $\alpha$-invariant is a morphism of (homotopy) ring spectra
\[\alpha\colon \MSpin\to \KO\]
inducing a ring homomorphism
\[ \alpha \colon \Omega_*^{\Spin} \to KO_* \]
on homotopy groups. In this section we present the construction of  a continuous map
\[c\colon G/O \to \Omega^\infty \MSpin\]
and give some properties. 
The interest of this construction is that, when precomposed with the inclusion $PL/O\to G/O$ and postcomposed with the $\alpha$-invariant, we obtain a group homomorphism
\[\alpha\colon \pi_* (PL/O) \to KO_*\]
which we will use to detect non-triviality of certain exotic spheres. Similarly, we may compose with the projection $G\to G/O$ to obtain a different group homomorphism (still denoted by the same letter)
\[\alpha\colon \pi_*^s\cong \pi_*(G) \to KO_*\]
where we use the isomorphism $\Jst\colon \pi_*G\cong \pi_*^s$ from  subsection \ref{sec:J-homomorphism}. 

We start by recalling some facts about orientations.  For a (homotopy) ring spectrum $\bfR$ and an $n$-dimensional  spherical fibration $p$ over $B$, with Thom space $T(p)$, an $\bfR$-orientation of $p$ is a choice of Thom class $\tau\in \bfR^n(T(p),*)$. The group of units $GL_1(\bfR^0(B))$ in the ring $\bfR^0(B)$ acts on the set of orientations by product and it is a consequence of the Thom isomorphism that this action is free and transitive. Thus if $\tau_1$ and $\tau_2$ are two $\bfR$-orientations of $p$, there is a difference class $\tau_1/\tau_2\in GL_1(\bfR^0(B))$, defined uniquely by the property that 
\[\tau_2 \cdot (\tau_1/\tau_2) = \tau_1\in \bfR^n(T(p),*).\]

As usual, let $GL_1(\bfR)\subset \Omega^\infty \bfR$ consist of those components which project to elements in $GL_1(\pi_0 \bfR)$. With this notation the difference class $\tau_2/\tau_1$ (just as any element in $GL_1(\bfR^0(B))$) is given by an unpointed homotopy class of maps $B\to GL_1(\bfR)$. (The space $GL_1(\bfS)$ is classically denoted by $F$ and the component of $GL_1(\mathbf{KO})$ containing the unit is classically denoted by $BO^\otimes$.)

In our setting $\bfR=\MSpin$. The universal bundle over $G_n/\Spin_n$ (classified by the projection to $B\Spin_n$) has two canonical $\MSpin$-orientations, $\tau_1$ (given by the spin structure) and $\tau_2$ (by the fiber homotopy trivialization). The difference classes $\tau_1/\tau_2$ for varying $n$ stabilize to yield a map
\[c\colon G/\Spin\to GL_1(\MSpin)\subset \Omega^\infty\MSpin.\]

An explicit description of this map is as follows: As the universal bundle $p_n$ over $G_n/\Spin_n$ pulls back from the universal bundle over $B\Spin_n$, there is an induced map on Thom spaces $T(p_n)\to\mathbf{MSpin}_n$. The fiber homotopy trivialization of $p_n$ induces a homotopy equivalence $T(p_n)\simeq \Sigma^n_+ (G_n/\Spin_n)$, so we get a map $\Sigma^n_+(G_n/\Spin_n)\to \mathbf{MSpin}_n$
adjoint to a map
\[c\colon G_n/\Spin_n\to \Omega^n \MSpin_n \subset \Omega^\infty \MSpin\]
of spaces. By construction, the base-point of $G_n/\Spin_n$ maps to the unit $1\in \pi_0 \MSpin$; it follows that $c$ takes values in $GL_1(\mathbf{MSpin})$. 

%\remws{Maybe discard explicit argument in final version?} 
% To see that this map $c$ describes the class $\tau_1/\tau_2$, we use the canonical fiber homotopy trivialization $\varphi$ of $p_n$ and note that $\varphi^*\tau_2\in \MSpin^n(S^n\wedge (G/\Spin)_+)$ is the canonical Thom class of the trivial bundle, corresponding to 1 under the suspension isomorphism. It follows that $\varphi^*\tau_1=(\tau_1/\tau_2)\cdot \varphi^* \tau_2$ corresponds to $\tau_1/\tau_2$ under the suspension isomorphism. Now recall that the $\MSpin$-Thom class $\tau_1$ is just the map induced on Thom spaces by classifying map of $p_n$ so $\varphi^*\tau_1$ is the same map pulled back with the homotopy equivalence $T(p_n)\simeq \Sigma^n_+(G_n/\Spin_n)$. But this is the definition of $c$.

\begin{remark}
If $G\Spin\to G$ denotes the 1-connected covering, then the composite
\[G\Spin/\Spin \to G/\Spin \to G/O\]
is a homotopy equivalence, hence a map on $G/\Spin$ gives rise to a map on $G/O$.
\end{remark}

% \begin{remark} Recall that $G/\Spin$, just as $G/O$ and $PL/O$ are infinite loop spaces, and we have by Madsen--Snaith--Tornehave \cite{MST} (see also \cite[V.7.10]{May}) that 
% \[\alpha\circ c\colon G/\Spin\to GL_1(\MSpin)\to GL_1(\KO)\]
% is an infinite loop map.
% \remws{It seems hard to believe that $c$ should not be an infinite loop map but this is really a result on $\KO$.}
% \end{remark}

The homotopy groups in positive degrees of $GL_1(\KO)$ are canonically identified with those of $KO$ by means of the canonical homotopy equivalence (of spaces)
\[GL_1(\KO)\simeq \Omega^\infty_0 \KO \times \{\pm 1\}.\]

% \begin{remark}
% This is a homotopy equivalence of spaces, not of infinite loop spaces. In particular the induced isomorphism of homotopy groups is not $\pi_*^s$-linear.
% \end{remark}

\remws{In the following, we denote the sphere spectrum by $\mathbf S$.}

\begin{lemma} \label{lem:alpha_commutes}
The map $\alpha\colon \pi_*(PL/O)\to KO_*$ is compatible with precomposition: If $x\in \pi_n(PL/O)$ and $f\colon S^{n+k}\to S^n$ with $k>0$, then  
\[ \alpha(x \circ f) = \alpha(x) \cdot Sf = \alpha(x) \cdot \alpha(Sf) \in KO_{n+k}\]
where $Sf\in \pi_k^s$ is the stabilization of $f$. In the middle term, multiplication by $f$ is through the unit map $\mathbf{S}\to \mathbf{KO}$ of the ring spectrum $\mathbf{KO}$ (given by $1\in\pi_0\mathbf{KO}$); in the last term we view $Sf$ as an element in $\pi_k(G)$ (through the isomorphism $I$ from above) which maps to an element in $\pi_k(G/\Spin)$, to which the $\alpha$-invariant may be applied.
\end{lemma}

\begin{proof}
As $\alpha$ is continuous, it is compatible with precomposition, and it is well-known that the action of $\pi_*^s$ by precomposition on the homotopy groups of a ring  spectrum agrees with the one through the unit map. This proves the  first identity.

Now recall that $GL_1(\mathbf{S})\simeq G$. Under this equivalence, $GL_1$ of the unit map $\mathbf{S}\to \MSpin$ factors through $c\colon G/\Spin\to GL_1(\MSpin)$ via the canonical projection $G\to G/\Spin$; this follows from the explicit description of $c$ as given above. As $\alpha\colon \MSpin\to \KO$ is a ring map, this implies  the second equality.
\end{proof}

% \comts{Will we use almost framed bordism in the final version? Should we keep
%   the following part if we don't?}
Finally, we would like to identify the map $c$, on homotopy groups, with the canonical map from almost framed bordism to spin bordism. Recall that a homotopy class of (unpointed) maps $f\colon S^n\to G/\Spin$ gives rise to a degree one normal map of a spin manifold $M^n_f$ onto the $n$-sphere, in particular to a spin bordism class $[M_f]\in\Omega_n^{\Spin}$, represented by a pointed homotopy class $c_f\colon S^n\to \Omega^\infty \MSpin$. 

On the other hand, the composite $c\circ f$ is an unpointed homotopy class $S^n\to \Omega^\infty\MSpin$. Since  the short exact sequence
\[0 \to \pi_n \MSpin \to [S^n, \Omega^\infty \MSpin] \to \pi_0 \MSpin\to 0\]
is canonically split, there is a canonical projection
\[\pi\colon [S^n, \Omega^\infty \MSpin]\to \pi_n \MSpin.\]

\begin{lemma}
We have $c_f=\pi(c\circ f)$ as pointed homotopy classes.
\end{lemma}

(Since $c\circ f$ lands in the 1-component of $\MSpin$, the pointed map $\pi(c\circ f)$ is represented by the loop space difference $c\circ f - 1$.)

\begin{proof}
We first note that, for $f\colon S^n\to G_k/\Spin_k$, the map $c\circ f$ is adjoint to the composite
\begin{equation}\label{eq:normal_invariants_in_MSpin}
S^k\wedge S^n_+ = T(\varepsilon^k) \to T(\gamma) \to \MSpin_{k}
\end{equation}
where $\gamma$ is the pull-back of the universal bundle $p_k$ along $f$ and the first map comes from the fiber homotopy trivialization classified by $f$. 

Now observe that the projection $\pi$ is characterized by the properties that it is the identity on the subgroup $\pi_n\MSpin$ and that it sends constant maps to zero. One can verify by inspection that these properties also hold for the composite
\[ [S^n, \Omega^{k} \MSpin_k]\cong [S^k\wedge S^n_+, \MSpin_k]_* \to  \pi_{n+k} \MSpin_k\]
where we pull back along along the Thom collapse $S^{n+k}\to S^k\wedge S^n_+$ of the $n$-sphere embedded trivially in $S^{n+k}$. Hence, the pointed class $\pi(c\circ f)$ is represented by pull-back of \eqref{eq:normal_invariants_in_MSpin} along the Thom collapse. But this composite is just the definition of $c_f$.
\end{proof}
%
% subsec:alpha (end)
%%%%%%%%%%%%%%%%%%%%%%%%%%%%%%%%%%%%%%%%%%%%%%%%

% sec:Hitchin (end)
%%%%%%%%%%%%%%%%%%%%%%%%%%%%%%%%%%%%%%%%%%%%%%%%
%%%%%%%%%%%%%%%%%%%%%%%%%%%%%%%%%%%%%%%%%%%%%%%%

% (fold)
\section{Toda brackets and homotopy spheres}
%%%%%%%%%%%%%%%%%%%%%%%%%%%%%%%%%%%%%%%%%%%%%%%%%%%%%%%%%%%%%%%%%%

%(fold)
\subsection{The $\alpha$-invariant on $\pi_*(PL_6/O_6)$} 
\label{subsec:alpha_on_PL6/O6}
%%%%%%%%%%%%%%%%%%%%%%%%%%%%%%%%%%%%%%%%%%%%%%%%%%%%%%%%%%%%%%%%%%

Recall that $\pi_{n+1} (PL/O)$ is identified, via smoothing theory, with  the group of 
%exotic 
\remdc{homotopy}
$(n{+}1)$-spheres. Denote by $\aelt_{PL/O}\in\pi_7(PL/O) \cong \Z/28$ the unique element of order 2. Also, for $j\geq 1$ and $\epsilon\in \{1,2\}$, let 
\[f \colon S^{8j-1+\epsilon}\to S^7\]
be any homotopy class such that 
\begin{itemize}
 \item $\alpha(f)=1$,
 \item $f$ is of order 2, and
 \item $f$ is the suspension of some $f'\in \pi_{8j-2+\epsilon}(S^6)$, equally of order 2.
\end{itemize}
In the case $\epsilon=1$, such elements $f$ exist for all  $j \geq 1$ by \cite[Theorem 1.2]{Adams}, where they are called $\mu_{8j+1}\in \pi_{8j+1}^s$. (Adams was mainly concerned about elements in the stable stems; in \cite[Lemma 2.14]{C-S} it was verified that the corresponding elements descend to order 2 elements on $S^7$, actually, on $S^5$.)
%Indeed Adams ambigously constructs such elements, calling them $\mu_{8j+1}$.  
%But since $\mu_{8j+1}$ refers to specific construction of Adams, we don't use this notation.
We can precompose any such element $f$ by the non-trivial element $\eta\in\pi_{8j+2}(S^{8j+1})$ to obtain a corresponding element for $\epsilon=2$, in view of Lemma \ref{lem:alpha_commutes} and the ring structure on $KO_*$.

As both $f$ and $\aelt_{PL/O}$ have order $2$, the Toda bracket 
\[ \an{f, 2, a_{PL/O}}\subset \pi_{8j+8+\epsilon}(PL/O) \] 
is defined. As explained in section \ref{subsec:alpha} we view the $\alpha$-invariant as a map 
$\pi_* (PL/O)\to KO_*$ by applying the canonical map $p\colon PL/O\to G/O$.

\begin{theorem} \label{thm:alpha_on_Toda}
%%%%%%%%%%%%%%
$\alpha\bigl( \an{f, 2, \aelt_{PL/O}} \bigr) = \{1\} \subset KO_{8j+\epsilon}=\{0,1\}.$
\end{theorem}

An ingredient in the proof of Theorem \ref{thm:alpha_on_Toda} is the following well-known lemma.

\begin{lemma}\label{lem:alpha_surjective}
The $\alpha$-invariant $\pi_8 (G/O) \to KO_8$ is surjective.
\end{lemma}

\begin{proof}
By its geometric definition, the $\alpha$-invariant of $x\in \pi_8 (G/O)$ in
$KO_{8}=\Z$ is calculated as 
the $\hat A$-class of the stable vector bundle \remts{over $S^8$} classified
by $x$. \remts{Hence 
  the $\alpha$-invariant
only depends on the image of $x$ in $\pi_{8}(BO)$, which is an infinite cyclic
group 
generated say by $t\in \pi_8(BO)\cong\mathbb{Z}$.} Now the image of
$\pi_{8}(G/O)$ in $\pi_{8}(BO)$ is precisely  the kernel of the
$J$-homomorphism, which is 
generated by $240 t$ \cite[6.26]{Lueck}. \remts{But the second Pontryagin class of $t$ in 
$H^8(S^8;\integers)$ is $\pm 6\cdot[S^8]$ by \cite{K} where we write $[S^8]$
for a generator of $H^8(S^8;\integers)\iso\integers$. Therefore the $\hat
A$-class of $t$ computes as 
$$\hat A_2(t)= \frac{1}{2^7\cdot 3^2\cdot 5}(-4
p_2(t)+7 p_1(t)^2)= -\frac{\pm 4\cdot 6 [S^8]}{2^7\cdot 3^2\cdot 5} =\mp
\frac{1}{240}[S^8],$$ compare \cite[p.\ 231]{LM}.}  \remts{Hence
the $\hat A$-class of $x=240t$ is equal to $\mp 1\cdot [S^8]$, i.e.~a
generator of $KO_8\cong
\integers$  is in the image of the $\alpha$-invariant, as we have claimed.}
\end{proof}

We will also use the  following well-known calculations from the surgery exact sequence for homotopy spheres (see e.g.\ \cite[Chapter 6]{Lueck}):
\begin{enumerate}
 \item\label{fact_I} $p_*\colon \pi_7 (PL/O)\to\pi_7 (G/O)$ is the zero map.
 \item\label{fact_II} $p_*\colon \pi_8 (PL/O)\to \pi_8 (G/O)$ is isomorphic to the inclusion $\Z/2\to \Z\oplus \Z/2$.
 \item\label{fact_III} $\pi_8 (G/PL)=\Z$ and $\pi_9(G/PL)=0$.
\end{enumerate}

\begin{proof}[Proof of Theorem \ref{thm:alpha_on_Toda}]
Since $\aelt_{PL/O}\in\pi_7(PL/O)$ is of order two, it  has a lift to some $\bar \aelt\in \pi_7^M(PL/O)$ which we can further map to $p(\bar\aelt)\in \pi_7^M(G/O)$. But $p(\aelt_{PL/O})=0\in\pi_7(G/O)$  by calculation \ref{fact_I}  so we may choose a lift of $p(a)$ to an element 
\[\delta(\aelt_{PL/O})\in\pi_8(G/O)\cong\Z\oplus\Z/2.\]
With these definitions, any element in the  Toda bracket $\an{f,2,\aelt_{PL/O}}$ is of the form $\bar\aelt\cdot \ol{Sf}$, hence its image in $G/O$ is of the  form
\[p(\bar\aelt)\cdot \ol{Sf}= \delta(\aelt_{PL/O})\cdot Sf\]
for a specific choice of $\delta(\aelt_{PL/O})$.

We proceed to calculate the $\alpha$-invariant of $\delta(\aelt_{PL/O})$. To do this, we first show that the residue class of $\delta(\aelt_{PL/O})$ in 
\[C:=\pi_8 (G/O)/(2, p_*\pi_8 (PL/O))\]
is non-trivial. Indeed, this residue class is precisely the image of 
%$\aelt_{PL/O}$ 
\remws{\[\aelt_{PL/O} \in \ker\bigl(p\colon \eot{\pi_7(PL/O)} \to \eot{\pi_7(G/O)}\bigr)\]}
under the %boundary 
\remws{connecting}  map for the snake lemma, applied to the diagram
\begin{equation}\label{eq:diagram_for_snake}
\xymatrix{
0 \ar[r] & \pi_8(PL/O)/2 \ar[r] \ar[d]^p & \pi_7^M(PL/O) \ar[r] \ar[d]^p & \eot{\pi_7(PL/O)} \ar[d]^p \ar[r] & 0\\
0 \ar[r] & \pi_8(G/O)/2 \ar[r] & \pi_7^M(G/O) \ar[r]& \eot{\pi_7(G/O)} \ar[r] & 0
}
\end{equation}
coming from \eqref{eq:ses_of_moore_groups}.

As a consequence of sequence \eqref{eq:ses_of_moore_groups} applied to $G/PL$ together with calculation \ref{fact_III}, we have $\pi_8^M(G/PL)=0$. By the Puppe sequence for the fibration sequence $PL/O\to G/O\to G/PL$, the middle vertical map in diagram \eqref{eq:diagram_for_snake} is therefore injective. The snake lemma implies that the  %boundary 
\remws{connecting map of  the snake lemma} is injective, so the residue class of $\delta(\aelt_{PL/O})$ in $C$ is non-zero as claimed.

Next we note that the $\alpha$-invariant $\pi_8(G/O)\to KO_8=\Z$ becomes zero after restricting 
%on 
\remdc{to}
the torsion group $\pi_8 (PL/O)$. So it induces a well-defined map 
\begin{equation}\label{eq:induced_alpha}
 C\to  KO_8/2=\Z/2,
\end{equation}
which, in view of Lemma \ref{lem:alpha_surjective}, is surjective. Indeed it is bijection as 
\remdc{Calculation \ref{fact_II} above} implies that $C \cong\Z/2$. 
We conclude that the $\alpha$-invariant of $\delta(\aelt_{PL/O})$ is odd. 

Now it follows from Lemma \ref{lem:alpha_commutes} that the $\alpha$-invariant of $\delta(\aelt_{PL/O})\cdot Sf$ in $KO_{8j+\epsilon}=\Z/2$ is non-zero, in view of  the ring structure  of $KO_*$ and our assumption %on $f$
\remws{that $\alpha(f)=1$}. But any element of the Toda bracket was of this form. 
\end{proof}

\begin{proof}[Proof of Theorem \ref{thm:PL/O}]
As
\[S_*\an{f,2,\aelt_{PL_6/O_6}} \subset \an{f,2,\aelt_{PL/O}}\subset \pi_{8j+\epsilon}(PL/O),\]
part (i) follows directly from Theorem \ref{thm:alpha_on_Toda}. We proceed to show (ii).

As pointed out in section \ref{subsec:Moore_space}, every element 
$g \in \an{f, 2, \aelt_{PL_6/O_6}}$ is realised as a composition
\[ g = \ol \aelt_{PL_6/O_6} \circ \ol{Sf} \colon S^{8j+\epsilon} \to M_7 \to X,\]
for specific choices of $\ol \aelt_{PL_6/O_6}$ and $\ol{Sf}$, where $\ol \aelt_{PL_6/O_6}$ extends $\aelt_{PL_6/O_6}$ over $M_7$, and 
$c \circ \ol{Sf} = Sf$ where $c \colon M_7 \to S^8$ is the map collapsing the $7$-cell to
a point. That is, $g$ is the image of $\ol \aelt_{PL_6/O_6}$ under the map
\begin{equation}\label{eq:pull_back_from_Moore}
(\ol{Sf})^* \colon \pi_7^M(PL_6/O_6) \to \pi_{8j+\epsilon}(PL_6/O_6),
\quad \ol b \mapsto \ol b \circ \ol{Sf}.
\end{equation}

Since
$f = Sf'$ where $2f'=0 \in \pi_{8j-2+\epsilon}(S^6)$, 
we can and do choose the map
$\overline{Sf}$ to be the suspension of a map 
$S^{8j-1+\epsilon}\to M_6$.  In this case the map $\ol{Sf} \colon S^{8j+\epsilon} \to M_7$ 
is a map of co-$H$-spaces and so $(\ol{Sf})^*$, defined in \eqref{eq:pull_back_from_Moore} above, 
is a group homomorphism. 

Now $\aelt_{PL_6/O_6}$ is divisible by 2 so $\eta\circ \aelt_{PL_6/O_6}=0$. It follows from Lemma \ref{lem:Moore_extension} that  every lift
$\bar \aelt_{PL_6/O_6}$ of $\aelt_{PL_6/O_6}$ has order 2.
It follows that $g$, as the homomorphic image of $\bar \aelt_{PL_6/O_6}$, has order 2.
\end{proof}

In Sections \ref{subsec:alpha_on_PL6} and \ref{sec:ImJ} 
below we will repeat the arguments of
Theorem \ref{thm:alpha_on_Toda} and the proof of Theorem \ref{thm:PL/O},
replacing $PL_6/O_6$ with other spaces.  To avoid repetition we summarise
these arguments as follows.  Let $h \colon X \to PL/O$ be a map.  Abusing
notation define $\alpha_X  := \alpha_{PL/O} \circ h \colon X\to KO$.
%
% \[ \remdc{
% \alpha_X  \colon X \xra{~~~h~~~} PL/O \xra{~\alpha_{PL/O~}} KO}
% \] 

\begin{theorem}\label{thm:meta-theorem}
% %%%%%%%%
Suppose that $a \in 2 \pi_7(X)$ satisfies $2a = 0$ and $h_*(a) \neq 0$.
% \in \pi_7(PL/O)$.
Then $\alpha_{X*} \colon \pi_*(X) \to KO_*$
is split onto in degrees $* \equiv 1, 2 \pmod 8$ and $\ast > 2$.
\end{theorem}

\begin{proof}
%%%%%%%
If $X$ is not simply connected we can and do replace $X$ by its universal cover.
Since $2a = 0$ we can form the Toda bracket
$\an{f, 2, a} \subset \pi_{8j+\epsilon}(X)$.  By naturality of Toda brackets and
the $\alpha$ invariant
\[ \alpha(\an{f, 2, a}) = \alpha(\an{f, a, h_*(a))} = \{1\}, \]
where the last equality holds by Theorem \ref{thm:alpha_on_Toda}.
The proof of Theorem \ref{thm:PL/O}
only used that $a_{PL_6/O_6} \in 2 \pi_7(PL_6/O_6)$,
so it may be repeated with $a \in 2 \pi_7(X)$ to show
that $\an{f, 2, a}$ contains an element of order two.
This completes the proof.
\end{proof}

% subsection alpha_on_PL6/06 end
%%%%%%%%%%%%%%%%%%%%%%%%%%%%%%%%%%%%%%%%%%%%%%%%

\subsection{The $\alpha$-invariant on $\pi_*(PL_6)$} \label{subsec:alpha_on_PL6}
%%%%%%%%%%%%%%%%%%%%%%%%%%%%%%%%%%%%%%%%%%%%%%%%
In this subsection we prove Theorem \ref{thm:PL6}, which states that the $\alpha$-invariant
\[ \alpha \colon \pi_{8j+\epsilon}(PL_m) \to KO_{8j+\epsilon} \]
is a split surjection for all $j \geq 1$, $m \geq 6$ and $\epsilon \in\{ 1, 2\}$.
Let $v \colon PL \to PL/O$ be the natural map
\remdc{and for $d \in \Z$, let $\rho_d \colon \Z \to \Z/d$ denote reduction mod~$d$.}
The following lemma is well-known.

\begin{lemma} \label{lem:pi7PL}
The homomorphism $v_* \colon \pi_7(PL) \to \pi_7(PL/O)$ is isomorphic to the surjection 
$\remdc{\rho_7 \oplus \id \colon \Z \oplus \Z/4 \xra{~~~} \Z/7 \oplus \Z/4.}$
\end{lemma}

\begin{proof}
The computation of $\pi_7(SPL) \cong \Z \oplus \Z/4$ 
is found in \cite[p.\,29]{Williamson}; see also \cite[Remark 4.9]{Brumfiel1}.
That $v_* \colon \pi_*(PL) \to \pi_*(PL/O)$ is onto follows from \cite[Theorem 3.1]{KM}
and \cite[Theorem 6.48]{Lueck}.
\end{proof}

\begin{lemma} \label{lem:SPL6}
%%%%%%%%%%%%%
The stablisation map $S_{*} \colon \pi_7(SPL_6) \to \pi_7(SPL)$
is isomorphic to the inclusion 
$
%i_4 \oplus {\rm Id}_{\Z/4}
\remdc{(\times 4)\oplus \id \colon \Z \oplus \Z/4 \xra{~~~} \Z \oplus \Z/4.}
$
\end{lemma}
\noindent
After we use it to prove Theorem \ref{thm:PL6},
the proof of Lemma \ref{lem:SPL6} will occupy the remainder of the subsection.

\begin{proof}[Proof of Theorem \ref{thm:PL6}]
%%%%%%%%%%%%%%%%%%%%%%%%
By Lemmas \ref{lem:pi7PL} and \ref{lem:SPL6},
the group $2\pi_7(SPL_6)$ has a unique element of order two
which maps to $a_{PL_6/O_6}$ under the composition
$PL_6 \to PL \to PL/O$.  The theorem now follows from
Theorem \ref{thm:meta-theorem}.
\end{proof}

For the proof of Lemma \ref{lem:SPL6} we require the following two lemmas.
They are presumably well-known; we include proofs for completeness.

\begin{lemma} \label{lem:pi7}
%%%%%%%%%%%%%%%
%Let $\rho_{240} \colon \Z \to \Z/240$ denote reduction mod~$240$.
The homomorphism $\pi_7(PL) \to \pi_7(G)$ is isomorphic to
the surjection 
\[ \remdc{ \rho_{240} + (\times 60) \colon \Z \oplus \Z/4 \xra{~~~} \Z/240.}  \]
\end{lemma}

\begin{proof}
%%%%%%%%%%%%%%%%%%%%%%
We consider the fibration $PL \to G\to G/PL$ and the following
part of its homotopy long exact sequence
\begin{equation*}
  \pi_8(G/PL)\to \pi_7(PL)\to \pi_7(G) \to \pi_7(G/PL).
\end{equation*}
Since $\pi_7(G/PL) = 0 = \pi_9(G/PL)$ and $\pi_8(G/PL) \cong \Z$ 
by surgery theory (see e.g.\ \cite[6.48]{Lueck}),
this sequence must be isomorphic to
%
% \begin{equation*}
% \integers \xra{~(60, -1)~} \integers\oplus \integers/4 
% \xra{~\rho_{240} \oplus (\times 60)} \integers/240 \to 0.
% \end{equation*}
%
\remws{\begin{equation*}
0\to \integers \to \integers\oplus \integers/4 
%\xra{(y \; z)} \integers/240 \to 0.
\xra{y+z} \integers/240 \to 0.
\end{equation*}
Since $y+z$ is surjective, $y$ is isomorphic to $\rho_{240}$ via an automorphism of $\Z/240$.
Since $y+z$ has a free kernel, $z$ must be injective and is thus isomorphic, via an automorphism of $\{0\} \oplus \Z/4$, to multiplication by 60 on residue classes.  The lemma now follows.}
\end{proof}

\begin{lemma} \label{lem:msJ}
%%%%%%%%%%%%%%%
For $k \leq 2n{-}5$
the homomorphism $\Junst \colon \pi_k(G_n) \to \pi_{k+n}(S^n)$ is an
isomorphism, for $k= 2n{-}4$ it is surjective.
\end{lemma}

\begin{proof}
Let $F_{n-1} \subset G_n$ be the submonoid of base-point preserving maps 
$S^{n-1} \to S^{n-1}$.  There is a fibration sequence $F_{n-1}\to SG_n \xra{{\rm ev}} S^{n-1}$,
where ${\rm ev}$ if given by evaluation at the base-point \cite[Lemma 3.1]{MM}.
It is well-known that the homotopy long exact sequences of these fibrations
fit into 
the following commutative diagram:
\[
\xymatrix@C=1.05em{
\pi_{k{+1}}(S^{n-1}) \ar[d]^{E^n} \ar[r] &
\pi_k(F_{n-1}) \ar[d]^{\cong} \ar[r] &
\pi_k(G_{n}) \ar[d]^{\Junst} \ar[r] &
\pi_k(S^{n-1}) \ar[d]^{E^n} \ar[r] &
\pi_{k-1}(F_{n-1}) \ar[d]^\cong
\\
\pi_{{n+k+1}}(S^{2n-1}) \ar[r] &
\pi_{n{-}1+k}(S^{n-1}) \ar[r]^(0.525){E} &
\pi_{n+k}(S^{n}) \ar[r]^(0.4){H} &
\pi_{n+k}(S^{2n-1}) \ar[r] &
\pi_{n+k-2}({S^{n-1}})
}\]
Here the maps labelled by ``$\cong$'' are the isomorphisms coming from the adjunction between based suspension and based loop space and the bottom row is part of the EHP sequence 
\cite[p.\,548]{Whitehead}.
{The vertical suspension maps $E^n$ are isomorphisms if $k \leq 2n-5$
and so $\Junst \colon \pi_k(G_n) \to \pi_{k+n}(S^n)$ is an isomorphism by the
$5$-Lemma. If $k= 2n-4$ only the right-most map $E^n$ is an isomorphism,
and the $5$-Lemma (rather the $4$-Lemma) implies that $I$ is surjective.}
\end{proof}

\begin{proof}[Proof of Lemma \ref{lem:SPL6}]
We compare $\pi_*(PL_6)$ with the homotopy groups of $\wt{PL_6}$\remws{, the semi-simplicial group of block automorphisms of  $\reals^6$.} By \cite[Proposition 5.6]{B-L}, the map 
$PL_6 \to \wt{PL}_6$ induces an isomorphism $\pi_7(PL_6) \to \pi_7(\wt{PL}_6)$.
Hence we consider the group $\pi_7(\wt{PL}_6)$ which lies in the 
following commutative diagram of exact sequences:
\[ \xymatrix{
0 \ar[r] &
\pi_8(G_6/ \wt{PL}_6) \ar[d]^\cong \ar[r] &
\pi_7(\wt{PL}_6) \ar[d] \ar[r] &
\pi_7(G_6) \ar[d] \ar[r] &
0\\
0 \ar[r] &
\pi_8(G/\wt{PL}) \ar[d]^{\cong}  \ar[r] &
\pi_7(\wt{PL}) \ar[d]^{\cong}  \ar[r] &
\pi_7(G) \ar[d]^{\cong} \ar[r] &
0\\
0 \ar[r] &
\Z \ar[r]^(0.4){(60, 1)} &
\Z \oplus \Z/4 \ar[r] &
\Z/240 \ar[r] &
0
} \]
The isomorphism between the bottom two sequences follows 
since in the limit $PL \to \wt{PL}$ is an equivalence \cite{B-L}
and the isomorphisms for $PL$ appeared 
in the proof of Lemma \ref{lem:pi7}.
Now the natural map
$\pi_8(G_6/\wt{PL}_6) \to \pi_8(G/\wt{PL})$ is an isomorphism (see e.g.~\cite[Theorem 1.10]{RS}).
Hence it suffices to prove that $i_* \colon \pi_7(G_6) \to \pi_7(G)$ 
is isomorphic to the inclusion $\Z/60 \to \Z/240$.
By Lemma \ref{lem:msJ}, the map
$\Junst \colon \pi_7(G_6) \to \pi_{13}(S^6)$ is an isomorphism.
It follows that the map $i_* \colon \pi_7(G_6) \to \pi_7(G)$ is isomorphic to the stabilisation
homomorphism $\pi_{13}(S^6) \to \pi_7^S$, which by \cite[Propositions 5.15 and 13.6]{Toda} is isomorphic to the inclusion $\Z/60 \to \Z/240$, as required.
%%%% basic algebra %%%%%
%To determine the extension $\pi_8(G_6/\wt{PL}_6) \to \pi_7(\wt{PL}_6) \to \pi_7(G_6)$,
%we note that it is classified by a class
%
%\[ \omega_6 \in H^2(\Z/60; \Z) \]
%
%which is given by $\omega_6 = \rho_{60}^*(\omega)$ where
%
%\[ \omega \in H^2(\Z/240; \Z) \]
%
%classifies $\pi_8(G/\wt{PL}) \to \pi_7(\wt{PL}) \to \pi_7(G)$, and 
%$\rho_{60}^*$ lies in the exact sequence
%
%\[ 0 \to H^2(\Z/4; \Z) \to H^2(\Z/240; \Z) \xra{\rho_{60}^*} H^2(\Z/60; \Z) \to 0 .\]
%
%Now by -ref, $\omega = 4 \sigma$, where $\sigma \in H^2(\Z/240; \Z)$ is 
%a generator, and so $\omega_6$ generates $H^2(\Z/15; \Z) \subset H^2(\Z/60; \Z)$.
%In particular, $\omega_6$ is odd-primary, and this completes the proof.
%
\end{proof}

%%%%% end subsection alpha on PL6 and TOP6
%%%%%%%%%%%%%%%%%%%%%%%%%%%%%%%%%%%%%%%%%%%%%
%

% sec:Toda and homotopy spheres (end)
%%%%%%%%%%%%%%%%%%%%%%%%%%%%%%%%%%%%%%%%%%%%%%%%
%%%%%%%%%%%%%%%%%%%%%%%%%%%%%%%%%%%%%%%%%%%%%%%%

% (fold)
\section{${\rm Im}(J)$-homotopy spheres} \label{sec:ImJ}
%%%%%%%%%%%%%%%%%%%%%%%%%%%%%%%%%%%%%%%%%%%%%%%%
%In this section we identify a subgroup $\Theta^J_{n+1} \subset \Theta_{n+1}$ of
%what we call $\im(J)$-homotopy spheres.
%We prove that all of the constructions for Section -ref can 
%be chosen to realise homotopy spheres in this subgroup.

In this section we prove Theorems \ref{thm:bP_spheres} and \ref{thm:ImJ_spheres},
both of which concern $\im(J)$-homotopy spheres.
The definition of ${\rm Im}(J)$-homotopy spheres is based on foundational facts
about the space $SG$ which we now recall.
For a prime $p$ and an $H$-space $X$, recall that $X_{(p)}$ denotes the $p$-localisation
of $X$.  % The following theorem combines parts of \cite[V Theorems 4.7 and 4.8]{May}.
The map $\phi \colon SG_{(p)} \times SG_{(p)} \to SG_{(p)}$ is the $p$-localisation
of the multiplication map on $SG$.

\begin{theorem}[{\cite[V Theorems 4.7 and 4.8]{May}}] \label{thm:SG_split}
For each prime $p$ there are spaces $\JS_p$ and $C_p$ and 
maps $i_{\JS_p} \colon \JS_p \to SG_{(p)}$ and $i_{C_p} \colon C_p \to SG_{(p)}$ such that the composition
\[  \JS_p \times C_p \xra{~i_{\JS_p} \times i_{C_p}} SG_{(p)} \times SG_{(p)} \xra{~\phi~} SG_{(p)} \]
is a weak homotopy equivalence. 
\end{theorem}

The homotopy groups of the spaces $\JS_p$ are closely related to the 
image of the $J$-homomorphism 
$I^{-1} \circ J_* \colon \pi_i(SO) \to \pi_*(SG)$ as we now recall.  
Let $\alpha_{J, p} \colon \pi_*(\JS_p) \to {\rm Tors}(KO_*)$
be the restriction of the $\alpha$-invariant on $\pi_*(SG)$ 
to $\pi_*(\JS_p) \subset \pi_*(SG)_{(p)}$.
The next lemma follows immediately from \cite[Remark 5.6]{May}.

\begin{lemma} \label{lem:ImJ1}
%%%%%%%%%%%%%%%%%
The groups $\im(J_*)_{(p)} \subset \pi_*(SG)_{(p)}$ are summands of the groups $\pi_*(\JS_p)$
and there is a split short exact sequence
\begin{equation}\label{eq:J_gorup}
 0 \to \im(J_*)_{(p)} \to \pi_*(\JS_p) \xra{~\alpha_{J, p}} {\rm
   Tors}(KO_*)_{(p)} \to 0.
\end{equation}

\end{lemma}

Following \cite[V \S4]{May}, we define 
$\JS_\infty: = {\prod_p}% substack{\bigtimes\\p}\,
\JS_p$ and $C_\infty : = {\prod_p}% substack{\bigtimes \\p}\,
C_p$
and 
% using Theorem \ref{thm:SG_split} we fix a weak equivalence
% 
% \begin{equation} \label{eq:psi}
% \psi \colon SG \xrightarrow{\sim} \JS_\infty \times C_\infty.
% \end{equation}
let
\begin{equation} \label{eq:psi}
\psi \colon SG \xrightarrow{\sim} \JS_\infty \times C_\infty
\end{equation}
be the weak equivalence stemming from Theorem \ref{thm:SG_split}.
%
%To avoid over-crowding the notation, 
We identify 
$\pi_*(SG) = \pi_*(\JS_\infty) \times \pi_*(C_\infty)$ using the map $\psi$ and then
define $\alpha_J \colon \pi_*(\JS_\infty) \to {\rm Tors}(KO_*)$ to be the restriction
of the $\alpha$-invariant on $\pi_*(SG)$ to $\pi_*(\JS_\infty)$.

Let $q \colon SG \to G/O$ be the natural map and observe that the isomorphism
$\Jst \colon \pi_*(SG) \xra{\cong} \pi_*^s$ induces an isomorphism
$\bar \Jst \colon {\rm Tors}(\pi_*(G/O)) \to \coker(\Jhom_*)$.
{The splitting $\pi_*(SG)=\pi_*(\JS_\infty)\times \pi_*(C_\infty)$ then induces a
splitting of $q_*$ and of its image as
\begin{equation}\label{eq:split_q}
  q_* = q^J_*\times q_*^C \colon \pi_*(\JS_\infty)\times \pi_*(C_\infty) \to
  q_*(\pi_*(\JS_\infty))\times q_*(\pi_*(C_\infty)) = \coker(J_*).
\end{equation}
Because $\im(J_*)$ is contained in $\pi_*(\JS_\infty)$ it follows that we have an isomorphism
$q_*^C\colon \pi_*(C_\infty)\to q_*(\pi_*(C_\infty))$, whereas $\alpha_J\colon
\pi_*(\JS_\infty)\to {\rm Tors(KO_*)}$ descends {by \eqref{eq:J_gorup}} to an isomorphism
$\overline{\alpha_J}\colon q_*(\pi_*(\JS_\infty))\to {\rm Tors}(KO_*)$, as
$\ker(\alpha_J)\subset\ker(q_*)$. We use the splitting
\eqref{eq:split_q} of $\coker(J_*)$, induced from the splitting \eqref{eq:psi}
of $SG$ to define a splitting of the $\alpha$-invariant on $\coker(J_*)$:
\begin{equation*}
  s_*:= {\rm incl} \circ \overline{\alpha_J}^{-1}\colon {\rm Tors}(KO_*)\to
  q_*(\pi_*(\JS_\infty))\into \coker(J_*).
\end{equation*}

}

%Next we identify $\coker(J_*)$ with the torsion subgroup of $\pi_*(G/O)$
%its image $q_*(\pi_*(SG)) \subset \pi_*(G/O)$
% = \bigoplus_p \pi_*(G/O)_{(p)}$
% We define a splitting $s_*$ of the $\alpha$-invariant on $\coker(J_*)$ by
% %
% \[ s_* \colon {\rm Tors}(KO_*) \to \coker(J_*),
% \quad x \mapsto \bar \Jst^{-1}(q_*(\alpha_{J}^{-1}(x))).\]
% %
% The map $s_*$ is well-defined since $\ker(\alpha_{J}) \subset \ker(q_*)$, and
% therefore also $\im(s_*)=\bar\Jst^{-1}\circ q_*(\pi_*(\JS_\infty))$.
Recalling the Kervaire-Milnor homomorphism $\Phi \colon \Theta_{n+1} \to \coker(J_{n+1})$ 
we make the following {definition.}

\begin{definition}[${\rm Im}(J)$-homotopy spheres]
%%%%%%%%%%%%%%%%%%
\label{def:ThetaJ}
A homotopy %$(n{+}1)$-
sphere $\Sigma \in \Theta_{n+1}$ is an {\em $\im(J)$-homotopy sphere} if
\[ \Phi(\Sigma) \in s_*({\rm Tors}(KO_{n+1})) \subset \coker(J_{n+1}) \]
and $\Theta^J_{n+1} \subset \Theta_{n+1}$
is the \emph{subgroup of $\im(J)$-homotopy spheres}. 
%The subgroup $\Gamma^{n+1, J}_{(k)} \subset \Gamma^{n+1}_{(k)}$ is defined by
%
%\[ \Gamma^{n+1, J}_{(k)} : = 
%\Gamma^{n+1}_{(k)} \cap \Sigma^{-1}(\Theta^J_{n+1}). \]
%
\end{definition}
 
\noindent
Since the Kervaire-Milnor
sequence $bP_{8k+2} \to \Theta_{8k+1} \to \coker(J_{8k+1})$ splits 
by \cite[Theorem 1.2]{Brumfiel2}, $bP_{8k+3} = 0$ and
${\rm Tors}(KO_{*}) = 0$ unless $* \equiv 1, 2$~mod~$8$,
we have

\begin{lemma} \label{lem:ImJ_and_Brumfiel}
%%%%%%%%%%%%%%%%%%%%%%
There is an isomorphism
$$\Theta^J_{n+1} \cong bP_{n+2} \oplus {\rm Tors}(KO_{n+1})$$
with $\alpha(\Theta^J_{n+1}) = {\rm Tors}(KO_{n+1})$ and $\alpha(bP_{n+2}) = 0$.
\end{lemma}

%We next define a space $SPL^J$ whose
%homotopy groups surject to $\Theta_{n+1}^J$.
%and which we shall need to prove our results concerning $\im(J)$-spheres.
Let $u \colon SPL \to SG$ be the natural map and 
let $\pr_{C_\infty} \colon SG \to C_{\infty}$ be the composition of 
the map $\psi$ of \eqref{eq:psi} and projection to the second factor.

\begin{definition}
We define $i^J \colon SPL^J \subset SPL$ to be the inclusion of the
homotopy fiber of the composition
\[ SPL \xra{~u~} SG \xra{~\pr_{C_\infty}~}  C_\infty.\]
Similarly we define $i^J_6 \colon SPL^J_6 \subset SPL_6$ to be the inclusion of the
homotopy fiber of the composition
\[ SPL_6 \xra{~S~} SPL \xra{~u~} SG \xra{~\pr_{C_\infty}~} C_\infty. \]
\end{definition}

\noindent
Let $v^+ \colon SPL \to PL/O$ be the restriction of $v \colon PL \to PL/O$ to $SPL$.

\begin{lemma} \label{lem:SPLJ_and_imJ}
  The image under the canonical maps of $\pi_*(SPL^J)$ consists precisely of
  the $\im(J)$-homotopy spheres, i.e.~we have
  \begin{equation*}
(\Psi \circ v^+_{*} \circ i^J_*)(\pi_{*}(SPL^J)) = \Theta_{*}^J.
\end{equation*}

\end{lemma}

\begin{proof}
{  We have that $\Phi\circ\Psi\circ v_*^+(\im(i_*^J))\subset
  q_*(\pi_*(\JS_\infty))$ by 
  naturality and because $u_*(\im(i_*^J))\subset \pi_*(\JS_\infty)$ by
  the splitting of $SG$. By definition of $\im(J)$-spheres therefore the left
  hand side is contained in the set of $\im(J)$-spheres.

  It remains to show that every $\im(J)$-sphere is contained in the left-hand
  side. First, we look at the summand $bP_{*+1}\subset \Theta^J_*$. Recall
  that the natural map $\pi_*(SPL)\to \pi_*(PL/O)$ is onto, corresponding to the fact that the
  stable tangent bundle of every homotopy sphere is trivial} 
  (see e.~g.~\cite[Theorem 6.45]{Lueck}). Every $bP$-sphere is
  mapped by $\Phi$ to $0$ in $\coker(J_*)$, therefore, using the splitting
  \eqref{eq:split_q} of $q_*$ and naturality, every lift of it to $\pi_*(SPL)$
  is mapped to $\pi_*(\JS_\infty)$ under $u_*$ and consequently lies in the image
  of $\pi_*(SPL^J)$.

  Because of Lemma \ref{lem:ImJ_and_Brumfiel}, it remains to find one sphere
  with 
  $\alpha$-invariant $1$ for each relevant dimension $8k{+}1$ and $8k{+}2$ in the
  left hand side.  We have to show that the restriction of the
  alpha-invariant to $u_*^{-1}(\pi_*(\JS_\infty))$ surjects onto ${\rm
    Tors(KO_*)}$. Note, however, that the cokernel of $\pi_*(SPL)\to
  \pi_*(SG)$ is the kernel of $\pi_*(SG)\to \pi_*(G/PL)$. {Via the
Kervaire-Milnor braid (see e.g.~\cite[Theorem 6.48]{Lueck})
the latter map can be identified with the Kervaire invariant which
  is known to be zero except for some dimensions $*=8k+6$ (compare \cite[Corollary 6.43]{Lueck}).} But these
  dimensions  are
  not relevant for us as in those dimensions ${\rm Tors}(KO_*)=0$. Therefore
  $u_*$ is surjective in the relevant dimensions, and because
  $\alpha(\pi_*(\JS_\infty))={\rm Tors(KO_*)}$ we are done.
\end{proof}

We require the following lemmas to prove Theorems \ref{thm:ImJ_spheres} and \ref{thm:bP_spheres}.
We defer their proofs to the end of the section.

\begin{lemma} \label{lem:SPLJ6}
%%%%%%%%%%%%%
The map $i^J_{6*} \colon \pi_7(SPL^J_6) \to \pi_7(SPL_6)$ 
is an isomorphism.
\end{lemma}

\begin{lemma}[C.f.~{\cite[Theorem 12.18]{Adams}}] \label{lem:ImJ}
Let $g \in \pi_{8j+1}(\JS_2)$ have $\alpha(g) = 1$.
Then
%the Toda bracket 
%
$$\an{\eta_{8j+1}, 2_{8j+1}, g} \subset \{2, 6\} \subset \pi_{8j+3}(\JS_2) \cong \Z/8.$$
% = \Z/8.$$
%
%constists of the generators of the subgroup of index two.
%The homomoprhism $\circ \eta_{8k+2} \colon \pi_{8k+2}(J_2) \to \pi_{8k+3}(J_2)$
%is isomorphic to the inclusion $\times 4 \colon \Z/2 \to \Z/8$.
\end{lemma}

\begin{proof}[Proof of Theorem \ref{thm:ImJ_spheres}]
%%%%%%%%%%%%%%	
From Lemmas \ref{lem:pi7PL}, \ref{lem:SPL6} and \ref{lem:SPLJ6} we deduce that there is an element $\aelt_{SPL_6^J} \in 2\pi_7(SPL_6^J)$ which maps
to $\aelt_{PL_6/O_6} \in \pi_7(PL_6/O_6)$ under the map induced by the composition
$SPL_6^J \to SPL_6 \to PL_6/O_6$.
The theorem follows from Theorem \ref{thm:meta-theorem}.
\end{proof}

\begin{proof}[Proof of Theorem \ref{thm:bP_spheres}]
%%%%%%%%%%%%%%
The proof of Theorem \ref{thm:ImJ_spheres}, 
shows that there is an element $g_{SPL_6^J}$  of order two
in the Toda bracket
$\an{f, 2, a_{SPL_6^J}} \subset \pi_{8j+1}(SPL_6^J)$,
so that $\alpha(g_{SPL_6^J}) = 1$.
Choose an element 
\[ e\in 
\an{\eta_{8j+1}, 2_{8j+1}, g_{SPL_6^J}} \subset \pi_{8j+3}(SPL_6^J) \]
and consider the following diagram:
\begin{equation} \label{eq:Im(J)-spheres}
\xymatrix{  
& & \pi_{8j+3}(SO) \ar[d] 
\\
\pi_{8j+3}(SPL_6^J) \ar[r]^{i^J_{6*}} & 
\pi_{8j+3}(SPL_6) \ar[d]^{q_{SPL_6}} \ar[r]^{S_{SPL*}} &
\pi_{8j+3}(SPL) \ar[d] \ar[r]^(0.45){\wh q_{SPL*}} &
\pi_{8j+3}(J_2 \times C_2) 
\\
&
\pi_{8j+3}(PL_6/O_6) \ar[r]^(0.525){S_{PL/O*}} & 
\pi_{8j+3}(PL/O)
} 
\end{equation}
By Lemma \ref{lem:ImJ} 
$$ \wh q_{SPL*} \circ S_{SPL*} \circ i_*(e)\in
 \{2, 6\} \oplus \{0\} \subset \pi_{8j+3}(J_2 \times C_2)$$
and by a theorem of Brumfiel  \cite[Theorem 1.4]{Brumfiel1}
$$\pi_{8j+3}(SPL) \cong \Z \oplus \Z/8 \oplus \coker(J_{8j+3}).$$
Since $e$ is $2$-primary torsion,
it follows that 
$$S_*(e)\in \{0\} \oplus \{2, 6\} \oplus \{0\}.$$
As $\pi_{8j+3}(SO) \cong \Z$ by Bott periodicity, and $\pi_{8j+3}(PL/O) \cong \Theta_{8j+3}$
is finite by \cite{KM}, the torsion of $\pi_{8j+3}(SPL)$ injects into 
$\pi_{8j+3}(PL/O)$.  Since the $\Z/8$-summand maps trivially to $\coker(J_{8j+3})$,
$e$ must map into the subgroup $\Psi(bP_{8j+4}) \subset \pi_{8j+3}(PL/O)$, and hence
to a generator of $\Psi(_4bP_{8j+3})$.
The commutativity of diagram \eqref{eq:Im(J)-spheres} above 
shows that $e\in\im(S_{PL_6/O_6*})$, which proves the theorem.
\end{proof}

\begin{proof}[Proof of Lemma \ref{lem:SPLJ6}]
%%%%%%%%%%%%%%	
%We first show that $S_* \colon \pi_7(SPL_6) \to \pi_7(SPL)$
%is isomorphic to the inclusion 
%$i_4 \oplus {\rm Id}_{\Z/4} \colon 4 \Z \oplus \Z/4 \hookrightarrow \Z \oplus \Z/4$.
%Then we show that $i^J_{*} \colon \pi_7(SPL^J_6) \to \pi_7(SPL_6)$ is an isomorphism.

To see that $i^J_{6*} \colon \pi_7(SPL^J_6) \to \pi_7(SPL_6)$ is
an isomorphism,
we recall that by the definition of $SPL_6^J$ we have a commutative diagram of fibrations
\[ \xymatrix{ 
F \ar[d] \ar[r]^{\cong} &
F \ar[d] \\
SPL^J_6 \ar[d] \ar[r] &
SPL_6 \ar[d] \\
J_2 \times \ast \ar[r] &
J_2 \times C_2
} \]
where $F$ is the homotopy fiber of the map $SPL_6 \to SPL \to \JS_2 \times C_2$.
Now $\pi_7^s = \im(\Jhom_7)$ and so $\pi_7(C_2) = 0$, and so diagram chasing in the ladder made by the homotopy long exact sequences of the above fibrations gives the result, provided that
we can prove that the map $\pi_8(SPL_6) \to \pi_8(\JS_2 \times C_2)$ is onto, and we do this now.
We have $\pi_8(\JS_2 \times C_2) \cong \pi_8^s \cong (\Z/2)^2$ and by \cite[Theorem 7.1]{Toda} the stablisation
homomorphism $\pi_{14}(S^6) \to \pi_8^s$ is onto.
By Lemma \ref{lem:msJ}
$J \colon \pi_8(G_6) \to \pi_{14}(S^6)$ is onto and 
by \cite[Proposition 5.6]{B-L} the map $\pi_8(\wt{PL}_6) \to \pi_8(PL_6)$ is onto.
Hence it is enough to show that $\pi_8(\wt{PL}_6) \to \pi_8(G_6)$ is onto.
But this follows from the exact sequence
\[ \dots \to \pi_8(\wt{PL}_6) \to \pi_8(G_6) \to \pi_8(G_6/\wt{PL}_6) \xra{~\del~} \pi_7(\wt{PL}_6) \to \dots ~, \]
{since the boundary map
$\del \colon \pi_8(G/\wt{PL}) \to \pi_7(\wt{PL}) \cong \pi_7(PL)$ is injective
and  $\pi_8(G_6/\wt{PL}_6) \cong \pi_8(G/\wt{PL})$
(we  
saw both assertions in the proof of Lemma \ref{lem:SPL6}).}
\end{proof}

%\begin{proof}[Proof of Lemma \ref{lem:pi7D}]
%%%%%%%%%%%%%%%%%%%%%%%
%The isomorphism of top rows was proven in Lemma \ref{lem:pi7}.
%That $\pi_7(G/O) = 0$ is well known -cite.
%That $v_* \colon \pi_*(PL) \to \pi_*(PL/O)$ is onto follows from \cite[Theorem 3.1]{KM}
%and \cite[Theorem 6.48]{Lueck},
%and this completes the proof.
%
%Here, we can use \cite[p.~29]{Williamson} for
%$\pi_7(PL)$. 
%
%\end{proof}

\begin{proof}[Proof of Lemma \ref{lem:ImJ}]
%%%%%%%%%%%%%%%%%%%%%%
In \cite[Proposition 12.18]{Adams} Adams proves that the $e$-invariant
of the Toda bracket $\an{\eta, 2, \mu_{8j+1}}$ is the set $\{ \frac{1}{4}, \frac{-1}{4} \} \in \Q/\Z$.
By \cite[Remark 5.6]{May} 
the $e$-invariant gives a split surjection from $(\pi_*^s)_{(2)}$ onto $\pi_*(J_2)$, 
proving the lemma.
\end{proof}

\begin{appendix}
\section{The Gromoll filtration: table of values} \label{sec:Gromoll}
%%%%%%%%%%%%%%%%%%%%%%%%%%%%%%%%%%%%%%%%%%%%%%%%

  We think that our results about the Gromoll filtration and the existence
  of elements rather deep down with non-trivial $\alpha$-invariant are
  interesting in their own right.  In this appendix we place them in context by
  assembling some results from the literature about the Gromoll
  filtration. This is an update of the corresponding table in \cite[Appendix
  A]{C-S}.  
  %We set
  \remdc{Recall
  $\Gamma^{4i-1}_{bP} = \Sigma^{-1}(bP_{4i}) \subseteq \Gamma^{4i-1}$, 
 let $f_M \in \Gamma^{4i-1}_{bP}$ be the generator corresponding to the Milnor sphere}
 and \remdc{define the group} 
 \remdc{$\Gamma^{4i-1}_{(k)\,bP}$} $: = \Gamma^{4i-1}_{(k)} \cap
  \Gamma^{4i-1}_{bP}$. 
  In the following table, the new results of the
current article are printed red.

\medskip\noindent
  \begin{tabular}{|l|p{7.8cm}|}
 \hline
  $\Gamma_{(5)}^7 \cong \Z/{28}$ & $\Gamma_{(5)}^7 \neq \Gamma_{(4)}^7 \supset 0 =
  \Gamma_{(3)}^7$. The inequality for $\Gamma_{(4)}^7 \neq \Gamma_{(5)}^7$ is
                                   due to Weiss 
    \cite{W2} who proved that $\Gamma_{(4)}^7$ has at most $14$ elements.\\
    \hline 
  $\Gamma_{(6)}^8 \cong \Z/2$ & nothing known %beyond $\Gamma_{(3)}^8 = 0$ by the Smale Conjecture.
  \\ \hline
  $\Gamma_{(7)}^9 \cong (\Z/2)^3$ & \textcolor{red}{$\Gamma^9_{(6)}\supset
                                    \Z/2$, $\alpha(\Gamma^9_{(6)})=\Z/2$} \emph{by Theorem \ref{thm:1}}%  
  \\ \hline
  $\Gamma_{(8)}^{10} \cong \Z/6$ & $\Gamma_{\textcolor{red}{(6)}}^{10}
                                    \supset \Z/2$, $\alpha(\Gamma^{10}_{\textcolor{red}{(6)}})=\Z/2$ \emph{by Theorem \ref{thm:1}}  %$\Gamma_{5}^{10} \supset \Z_3$
  \\ \hline
  $\Gamma_{(9)}^{11} \cong \Z/{992}$ & $\Gamma^{11}_{(8)} \subset \Z/{496}$ by
                                       \cite{W1},
                                       \textcolor{red}{$\Gamma^{11}_{(6)} \supset
                                       \Z/4$} \emph{by Theorem \ref{thm:bP_spheres}}
  \\ \hline
  $\Gamma_{(10)}^{12} = 0$ &\\ \hline
  $\Gamma_{(11)}^{13} \cong \Z/3$ & $\Gamma_{(11)}^{13} = \Gamma_{(10)}^{13} =
  \Gamma_{(9)}^{13}$ by \cite{A-B-K} % $\Gamma_{8}^{13} \cong \Z_3$
  \\ \hline
  $\Gamma_{(12)}^{14} \cong \Z/2$ & nothing known\\ \hline
  $\Gamma_{(13)}^{15} \cong \Z/2 \oplus \Z/{8128}$ & $\Gamma_{(12)}^{15} \cong \Z/2
  \oplus \Z/{4064}$ by \cites{A-B-K, W1}\\ \hline
  % $\wt \Gamma_{9}^{15} \cong \Z_2$.
  $\Gamma_{(14)}^{16} \cong \Z/2$ &nothing known,  conjecturally $\Gamma_{(13)}^{16}
  = 0$\\ \hline
   $\Gamma_{(15)}^{17} \cong (\Z/2)^2$ &
                                         \textcolor{red}{$\Gamma^{17}_{(6)}\supset
                                         \Z/2$, $\alpha(\Gamma^{17}_{(6)})=\Z/2$} \emph{by Theorem \ref{thm:1}}\\
   \hline 
   $\Gamma_{(16)}^{18} \cong \Z/8 \oplus \Z/2$ & \emph{By Theorem \ref{thm:1}},
  $\alpha(\Gamma^{18}_{\textcolor{red}{(6)}})=\Z/2$. Because $\integers/8=\ker(\alpha)$,
   $\Gamma^{18}_{\textcolor{red}{(6)}} \supset \{0\}\oplus \Z/2$.\\ \hline
   $\Gamma^{8j+1}$,  $j\ge 1$  & \textcolor{red}{$\Gamma_{(6)}^{8j+1} \supset \Z/2$, $\alpha(\Gamma^{8j+1}_{(6)})=\Z/2$} \emph{by Theorem
    \ref{thm:1}}\\ \hline
   $\Gamma^{8j+2}$,   $j\ge 1$  & $\Gamma_{ \textcolor{red}{(6)}}^{8j+2} \supset \Z/2$, $\alpha(\Gamma_{ \textcolor{red}{(6)}}^{8j+2}) = \Z/2$ \emph{by Theorem
    \ref{thm:1}}\\ \hline
$\Gamma^{8j+3}_{bP}$,  $j\ge 1$    & 
% \textcolor{red}{$_{bP}\Gamma_{(6)}^{8j+3}\supset \Z/4$} 
\textcolor{red}{$\Gamma_{(6)\,bP}^{8j+3}  \supset \Z/4$} 
\emph{by Theorem
     \ref{thm:bP_spheres}} \\ \hline
       %%%%%% earlier general results on bP_{4k} spheres %%%%%%%
   $\Gamma^{4i-1}_{bP}, i \geq 4$ & $\remdc{\Gamma^{4i-1}_{(2i+1)\,bP} \neq 0}$
   by \cite[Theorem 1.1]{A-B-K} \\ \hline
   $\Gamma^{4i-1}_{bP}, i \geq 2$ & $f_M \notin$ \remdc{$\Gamma^{4i-1}_{(4i-4)\, bP}$} by 
   \cite[2nd Corollary, p.\,888]{W1} \\ \hline
 
   %%%%%%% our results
  % $\Gamma_{2}^{19} \cong \Z/2 \oplus bP_{20}$
  % $\Gamma_{2}^{20} \cong \Z/8 \oplus \Z/3$
  % $\Gamma_{2}^{21} \cong \Z/2^3$
  % $\Gamma_{2}^{22} \cong \Z/2^2$
 % $\Gamma_{2}^{23} \cong \Z/8 \oplus \Z/2 \oplus \Z/3 \oplus bP_{24}$
 \end{tabular}

%%%%%%%%%%%%%%%%%%%%%%%%%%%%%%%%%%%%%%%%%%%%%%%%
\end{appendix}


\begin{bibdiv}
  \begin{biblist}
    \bib{Adams}{article}{
    AUTHOR = {Adams, J. F.},
     TITLE = {On the groups {$J(X)$}. {IV}},
   JOURNAL = {Topology},
    VOLUME = {5},
      YEAR = {1966},
     PAGES = {21--71},
      ISSN = {0040-9383},
  review   = {\MR{0198470}},
}

\bib{AmmannDahlHumbert_adv}{article}{
   author={Ammann, Bernd},
   author={Dahl, Mattias},
   author={Humbert, Emmanuel},
   title={Surgery and harmonic spinors},
   journal={Adv. Math.},
   volume={220},
   date={2009},
   number={2},
   pages={523--539},
   issn={0001-8708},
   review={\MR{2466425}},
   doi={10.1016/j.aim.2008.09.013},
}
\bib{AmmannDahlHumbert2}{article}{
   author={Ammann, Bernd},
   author={Dahl, Mattias},
   author={Humbert, Emmanuel},
   title={Harmonic spinors and local deformations of the metric},
   journal={Math. Res. Lett.},
   volume={18},
   date={2011},
   number={5},
   pages={927--936},
   issn={1073-2780},
   review={\MR{2875865}},
   doi={10.4310/MRL.2011.v18.n5.a10},
}
		
\bib{A-B-K}{article}{
   author={Antonelli, P.},
   author={Burghelea, D.},
   author={Kahn, P. J.},
   title={Gromoll groups, ${\rm Diff}S\,^{n}$ and bilinear constructions
   of exotic spheres},
   journal={Bull. Amer. Math. Soc.},
   volume={76},
   date={1970},
   pages={772--777},
   issn={0002-9904},
   review={\MR{0283809}},
}
\bib{Baer}{article}{
   author={B{\"a}r, Christian},
   title={Metrics with harmonic spinors},
   journal={Geom. Funct. Anal.},
   volume={6},
   date={1996},
   number={6},
   pages={899--942},
   issn={1016-443X},
   review={\MR{1421872}},
   doi={10.1007/BF02246994},
}
\bib{BotvinnikEbertRandal-Williams}{article}{
  author={Botvinnik, Boris},
  author={Ebert, Johannes},
  author={Randal-Williams, Oscar},
  title={Infinite loop spaces and positive scalar curvature},
  journal={Invent. Math.},
  pages={1--87},
  year={2017},
  publisher={Springer}
}
 
%  \bib{Botvinnik-Gilkey}{article}{
%     author={Botvinnik, Boris},
%     author={Gilkey, Peter B.},
%     title={The eta invariant and metrics of positive scalar curvature},
%     journal={Math. Ann.},
%     volume={302},
%     date={1995},
%     number={3},
%     pages={507--517},
%     issn={0025-5831},
%     review={\MR{1339924 (96f:58159)}},
%     doi={10.1007/BF01444505},
%  }
% \bib{BHSW}{article}{
%    author={Botvinnik, Boris},
%    author={Hanke, Bernhard},
%    author={Schick, Thomas},
%    author={Walsh, Mark},
%    title={Homotopy groups of the moduli space of metrics of positive scalar
%    curvature},
%    journal={Geom. Topol.},
%    volume={14},
%    date={2010},
%    number={4},
%    pages={2047--2076},
%    issn={1465-3060},
%    review={\MR{2680210}},
%    doi={10.2140/gt.2010.14.2047},
% }
\bib{Brumfiel1}{article}{
   author={Brumfiel, G.},
   title={On the homotopy groups of ${\rm BPL}$ and ${\rm PL/O}$},
   journal={Ann. of Math. (2)},
   volume={88},
   date={1968},
   pages={291--311},
   issn={0003-486X},
   review={\MR{0234458}},
}
\bib{Brumfiel2}{article}{
   author={Brumfiel, G.},
   title={On the homotopy groups of $BPL$ and $PL/O$. II},
   journal={Topology},
   volume={8},
   date={1969},
   pages={305--311},
   issn={0040-9383},
   review={\MR{0248830}},
}			
\bib{B-L}{article}{
   author={Burghelea, Dan},
   author={Lashof, Richard},
   title={The homotopy type of the space of diffeomorphisms. I, II},
   journal={Trans. Amer. Math. Soc.},
   volume={196},
   date={1974},
   pages={1--36; ibid. 196\ (1974), 37--50},
   issn={0002-9947},
   review={\MR{0356103}},
}
% \bib{Carr}{article}{
%    author={Carr, Rodney},
%    title={Construction of manifolds of positive scalar curvature},
%    journal={Trans. Amer. Math. Soc.},
%    volume={307},
%    date={1988},
%    number={1},
%    pages={63--74},
%    issn={0002-9947},
%    review={\MR{936805}},
%    doi={10.2307/2000751},
% }

% \bib{Ce}{article}{
%    author={Cerf, Jean},
%    title={La stratification naturelle des espaces de fonctions
%    diff\'erentiables r\'eelles et le th\'eor\`eme de la pseudo-isotopie},
%    language={French},
%    journal={Inst. Hautes \'Etudes Sci. Publ. Math.},
%    number={39},
%    date={1970},
%    pages={5--173},
%    issn={0073-8301},
%    review={\MR{0292089}},
% }
\bib{C-S}{article}{
   author={Crowley, Diarmuid},
   author={Schick, Thomas},
   title={The Gromoll filtration, $KO$-characteristic classes and metrics of
   positive scalar curvature},
   journal={Geom. Topol.},
   volume={17},
   date={2013},
   number={3},
   pages={1773--1789},
   issn={1465-3060},
   review={\MR{3073935}},
   doi={10.2140/gt.2013.17.1773},
}

		
\bib{Ebert}{article}{
   author={Ebert, Johannes},
   title={The two definitions of the index difference},
    journal={Trans. Americ. Math. Soc.},
  volume={369},
  number={10},
  pages={7469--7507},
  year={2017},
}		
\bib{G-RW}{article}{
   author={Galatius, S{\o}ren},
   author={Randal-Williams, Oscar},
   title={Stable moduli spaces of high-dimensional manifolds},
   journal={Acta Math.},
   volume={212},
   date={2014},
   number={2},
   pages={257--377},
   issn={0001-5962},
   review={\MR{3207759}},
   doi={10.1007/s11511-014-0112-7},
}
	
\bib{G}{article}{
   author={Gromoll, Detlef},
   title={Differenzierbare Strukturen und Metriken positiver Kr\"ummung auf
   Sph\"aren},
   language={German},
   journal={Math. Ann.},
   volume={164},
   date={1966},
   pages={353--371},
   issn={0025-5831},
   review={\MR{0196754}},
 }
%  \bib{HaefligerWall}{article}{
%    author={Haefliger, A.},
%    author={Wall, C. T. C.},
%    title={Piecewise linear bundles in the stable range},
%    journal={Topology},
%    volume={4},
%    date={1965},
%    pages={209--214},
%    issn={0040-9383},
%    review={\MR{0184243}},
% }
	
\bib{HSS}{article}{
   author={Hanke, Bernhard},
   author={Schick, Thomas},
   author={Steimle, Wolfgang},
   title={The space of metrics of positive scalar curvature},
   journal={Publ. Math. Inst. Hautes \'Etudes Sci.},
   volume={120},
   date={2014},
   pages={335--367},
   issn={0073-8301},
   review={\MR{3270591}},
   doi={10.1007/s10240-014-0062-9},
}
		
\bib{H}{article}{
   author={Hitchin, Nigel},
   title={Harmonic spinors},
   journal={Advances in Math.},
   volume={14},
   date={1974},
   pages={1--55},
   issn={0001-8708},
   review={\MR{0358873}},
}

\bib{K}{article}{
   author={Kervaire, Michel A.},
   title={A note on obstructions and characteristic classes},
   journal={Amer. J. Math.},
   volume={81},
   date={1959},
   pages={773--784},
   issn={0002-9327},
   review={\MR{0107863}},
}
	\bib{KM}{article}{
   author={Kervaire, Michel A.},
   author={Milnor, John W.},
   title={Groups of homotopy spheres. I},
   journal={Ann. of Math. (2)},
   volume={77},
   date={1963},
   pages={504--537},
   issn={0003-486X},
   review={\MR{0148075}},
}
\bib{KS}{book}{
   author={Kirby, Robion C.},
   author={Siebenmann, Laurence C.},
   title={Foundational essays on topological manifolds, smoothings, and
   triangulations},
   note={With notes by John Milnor and Michael Atiyah;
   Annals of Mathematics Studies, No. 88},
   publisher={Princeton University Press, Princeton, N.J.; University of
   Tokyo Press, Tokyo},
   date={1977},
   pages={vii+355},
   review={\MR{0645390}},
}

% \bib{Kupers}{unpublished}{
%    author={Alexander Kupers},
%    title={Some finiteness results for groups of automorphisms of manifolds},
%    note={in preparation, Copenhagen},
%    year={2016},
% }

\bib{Lance}{article}{
   author={Lance, Timothy},
   title={Differentiable structures on manifolds},
   conference={
      title={Surveys on surgery theory, Vol. 1},
   },
   book={
      series={Ann. of Math. Stud.},
      volume={145},
      publisher={Princeton Univ. Press, Princeton, NJ},
   },
   date={2000},
   pages={73--104},
   review={\MR{1747531}},
}
			
\bib{LM}{book}{
   author={Lawson, H. Blaine, Jr.},
   author={Michelsohn, Marie-Louise},
   title={Spin geometry},
   series={Princeton Mathematical Series},
   volume={38},
   publisher={Princeton University Press, Princeton, NJ},
   date={1989},
   pages={xii+427},
   isbn={0-691-08542-0},
   review={\MR{1031992}},
}
		

\bib{Lueck}{article}{
   author={L{\"u}ck, Wolfgang},
   title={A basic introduction to surgery theory},
   conference={
      title={Topology of high-dimensional manifolds, No. 1, 2},
      address={Trieste},
      date={2001},
   },
   book={
      series={ICTP Lect. Notes},
      volume={9},
      publisher={Abdus Salam Int. Cent. Theoret. Phys., Trieste},
   },
   date={2002},
   pages={1--224},
   review={\MR{1937016}},
}

\bib{MM}{book}{
   author={Madsen, Ib},
   author={Milgram, R. James},
   title={The classifying spaces for surgery and cobordism of manifolds},
   series={Annals of Mathematics Studies},
   volume={92},
   publisher={Princeton University Press, Princeton, N.J.; University of
   Tokyo Press, Tokyo},
   date={1979},
   pages={xii+279},
   isbn={0-691-08225-1},
   review={\MR{548575}},
}
		
% \bib{MST}{article}{
%    author={Madsen, I.},
%    author={Snaith, V.},
%    author={Tornehave, J.},
%    title={Infinite loop maps in geometric topology},
%    journal={Math. Proc. Cambridge Philos. Soc.},
%    volume={81},
%    date={1977},
%    number={3},
%    pages={399--430},
%    issn={0305-0041},
%    review={\MR{0494076}},
% }


\bib{May}{book}{
   author={May, J. Peter},
   title={$E_{\infty }$ ring spaces and $E_{\infty }$ ring spectra},
   series={Lecture Notes in Mathematics, Vol. 577},
   note={With contributions by Frank Quinn, Nigel Ray, and J\o rgen
   Tornehave},
   publisher={Springer-Verlag, Berlin-New York},
   date={1977},
   pages={268},
   review={\MR{0494077}},
}

\bib{May2}{article}{
   author={May, J. P.},
   title={The spectra associated to ${\cal I}$-monoids},
   journal={Math. Proc. Cambridge Philos. Soc.},
   volume={84},
   date={1978},
   number={2},
   pages={313--322},
   issn={0305-0041},
   review={\MR{0488033}},
}

\bib{M}{article}{
   author={Milnor, John W.},
   title={Remarks concerning spin manifolds},
   conference={
      title={Differential and Combinatorial Topology (A Symposium in Honor
      of Marston Morse)},
   },
   book={
      publisher={Princeton Univ. Press, Princeton, N.J.},
   },
   date={1965},
   pages={55--62},
   review={\MR{0180978}},
}
\bib{Morlet}{article}{
   author={Morlet, Claude},
   title={Isotopie et pseudo-isotopie},
   language={French},
   journal={C. R. Acad. Sci. Paris S\'er. A-B},
   volume={266},
   date={1968},
   pages={A559--A560},
   review={\MR{0236935}},
}
		
\bib{N}{article}{
   author={Neisendorfer, Joseph A.},
   title={Homotopy groups with coefficients},
   journal={J. Fixed Point Theory Appl.},
   volume={8},
   date={2010},
   number={2},
   pages={247--338},
   issn={1661-7738},
   review={\MR{2739026}},
   doi={10.1007/s11784-010-0020-1},
}		
% \bib{PSpsc}{article}{
%    author={Piazza, Paolo},
%    author={Schick, Thomas},
%    title={Groups with torsion, bordism and rho invariants},
%    journal={Pacific J. Math.},
%    volume={232},
%    date={2007},
%    number={2},
%    pages={355--378},
%    issn={0030-8730},
%    review={\MR{2366359 (2008j:58035)}},
%    doi={10.2140/pjm.2007.232.355},
% }
\bib{Rosenberg}{article}{
   author={Rosenberg, Jonathan},
   title={$C^{\ast} $-algebras, positive scalar curvature, and the Novikov
   conjecture},
   journal={Inst. Hautes \'Etudes Sci. Publ. Math.},
   number={58},
   date={1983},
   pages={197--212 (1984)},
   issn={0073-8301},
   review={\MR{720934 (85g:58083)}},
}
\bib{RS}{article}{
   author={Rourke, C. P.},
   author={Sanderson, B. J.},
   title={Block bundles. III. Homotopy theory},
   journal={Ann. of Math. (2)},
   volume={87},
   date={1968},
   pages={431--483},
   issn={0003-486X},
   review={\MR{0232404}},
}
\bib{Sch}{article}{
   author={Schick, Thomas},
   title={A counterexample to the (unstable) Gromov-Lawson-Rosenberg
   conjecture},
   journal={Topology},
   volume={37},
   date={1998},
   number={6},
   pages={1165--1168},
   issn={0040-9383},
   review={\MR{1632971 (99j:53049)}},
   doi={10.1016/S0040-9383(97)00082-7},
 }
    \bib{ICM}{inproceedings}{
   author={Schick, Thomas}, title={The topology of scalar curvature},
  booktitle={Proceedings of the International Congress of Mathematicians Seoul
    2014, VOLUME II},
  pages={1285-1308},
note={arXiv:1405.4220},
  year={2014},}

\bib{Stolz}{article}{
   author={Stolz, Stephan},
   title={Simply connected manifolds of positive scalar curvature},
   journal={Ann. of Math. (2)},
   volume={136},
   date={1992},
   number={3},
   pages={511--540},
   issn={0003-486X},
   review={\MR{1189863 (93i:57033)}},
   doi={10.2307/2946598},
}
\bib{Toda}{book}{
   author={Toda, Hirosi},
   title={Composition methods in homotopy groups of spheres},
   series={Annals of Mathematics Studies, No. 49},
   publisher={Princeton University Press, Princeton, N.J.},
   date={1962},
   pages={v+193},
   review={\MR{0143217}},
}
\bib{Waterstraat}{article}{
   author={Waterstraat, Nils},
   title={A remark on the space of metrics having nontrivial harmonic
   spinors},
   journal={J. Fixed Point Theory Appl.},
   volume={13},
   date={2013},
   number={1},
   pages={143--149},
   issn={1661-7738},
   review={\MR{3071945}},
   doi={10.1007/s11784-013-0096-5},
}
\bib{W1}{article}{
   author={Weiss, Michael},
   title={Sph\`eres exotiques et l'espace de Whitehead},
   language={French, with English summary},
   journal={C. R. Acad. Sci. Paris S\'er. I Math.},
   volume={303},
   date={1986},
   number={17},
   pages={885--888},
   issn={0249-6291},
   review={\MR{870913 (87m:57038)}},
}	
\bib{W2}{article}{
   author={Weiss, Michael},
   title={Pinching and concordance theory},
   journal={J. Differential Geom.},
   volume={38},
   date={1993},
   number={2},
   pages={387--416},
   issn={0022-040X},
   review={\MR{1237489 (95a:53057)}},
}

\bib{W4}{unpublished}{
   author={Weiss, Michael},
   title={Dalian notes on rational Pontrjagin classes},
   note={arXiv: http://arxiv.org/pdf/1507.00153v3.pdf},
   year={2016},
}

\bib{Whitehead}{book}{
   author={Whitehead, George W.},
   title={Elements of homotopy theory},
   series={Graduate Texts in Mathematics},
   volume={61},
   publisher={Springer-Verlag, New York-Berlin},
   date={1978},
   pages={xxi+744},
   isbn={0-387-90336-4},
   review={\MR{516508}},
}

\bib{Williamson}{article}{
   author={Williamson, Robert E., Jr.},
   title={Cobordism of combinatorial manifolds},
   journal={Ann. of Math. (2)},
   volume={83},
   date={1966},
   pages={1--33},
   issn={0003-486X},
   review={\MR{0184242}},
}

			
  \end{biblist}
\end{bibdiv}
\end{document}